\documentclass[11pt]{amsart}
\usepackage[left=3.2cm,right=3.2cm,top=3.5cm,bottom=3.cm]{geometry}
\usepackage[T1]{fontenc}
\usepackage[utf8]{inputenc}
\usepackage[english]{babel}
\usepackage{graphicx, amsfonts, amsthm, amsxtra, amssymb, verbatim,latexsym,bbm}
\usepackage{calrsfs}
\usepackage{addfont}
\usepackage{hyperref}
\usepackage[mathscr]{euscript}
\usepackage{color}   
\usepackage[all,cmtip]{xy}
\usepackage{enumitem}
\usepackage{lmodern}

\def\Z{\mathbb{Z}}                   
\def\Q{\mathbb{Q}}                   
\def\C{\mathbb{C}}                   
\def\N{\mathbb{N}}                   
\def\R{\mathbb{R}}                   

\def\rvf{W}   
\def\cvf{\delta}   
\def\di2{{ m}}   

\newcommand{\red}{}
\newcommand{\Res}{\mathrm{Res}}
\newcommand{\pro}{\mathbb{P}}
\newcommand{\Po}{\mathbb{H}}
\newcommand{\SL}{\mathrm{SL}}

\def\={\;=\;}
\def\+{\;+\;}
\def\-{\;-\;}
\def\:={\;:=\;}

\def\sl2{\mathfrak{sl}_2(\C)}
\def\SL2{{\rm SL}_2(\Z)}
\def\qmfs{\widetilde{\mathscr{M}}}             
\def\mfs{{\mathscr{M}}}                            

\def\rsdo{{\partial}}             

\def\G0g{{\Gamma_0(N_g)}}             
\def\wss{{{\sf w}}}             

\def\vs{{t}}             
\def\prs{{\mathscr{P}}}             
\def\nvs{{d}}             

\def\ew4{{\Lambda}}             

\def\mods{{RRC system}}
\newtheorem{theo}{Theorem}[section]

\newtheorem{coro}{Corollary}[section]

\newtheorem{prop}{Proposition}[section]
\newtheorem{rem}{Remark}[section]

\numberwithin{equation}{section}

\newtheorem*{theorem*}{Theorem}
\newtheorem{lemma}{Lemma}
\newtheorem{proposition}{Proposition}

\title{Ramanujan systems of Rankin-Cohen type and hyperbolic triangles}
\author[G.\,Bogo]{Gabriele Bogo}
\address{Fachbereich Mathematik\\ Technische Universität Darmstadt\\ Schlossgartenstrasse 7, 64289 Darmstadt\\
Germany}
\email{bogo@mathematik.tu-darmstadt.de}
\author[Y.\,Nikdelan]{Younes Nikdelan}
\address{Departamento de Análise Matemática\\ Instituto de Matemática e Estatística (IME)\\ Universidade do Estado do Rio de Janeiro (UERJ)\\ Rua São Francisco Xavier, 524, Rio de Janeiro, Brazil/ CEP: 20550-900}
\email{younes.nikdelan@ime.uerj.br}
\date{}
\subjclass[2020]{11F03, 34A34 (primary), 16W50, 11F55 (secondary)}
\keywords{Modular forms, systems of nonlinear ODEs, Rankin-Cohen brackets, Triangle groups, Modular embeddings}

\begin{document}
\maketitle
\begin{abstract}
In the first part of the paper we characterize certain systems of first order nonlinear differential equations whose space of solutions is an $\mathfrak{sl}_2(\C)$-module. We prove that such systems, called~\emph{Ramanujan systems of Rankin-Cohen type}, have a special shape and are precisely the ones whose solution space admits a Rankin-Cohen structure.
In the second part of the paper we consider triangle groups~$\Delta(n,m,\infty)$. By means of modular embeddings, we associate to every such group a number of systems of nonlinear ODEs whose solutions are algebraically independent twisted modular forms. In particular, all rational weight modular forms on~$\Delta(n,m,\infty)$ are generated by the solutions of one such system (which is of Rankin-Cohen type). As a corollary we find new relations for the Gauss hypergeometric function evaluated at functions on the upper half-plane. To demonstrate the power of our approach in the non-classical setting, we construct the space of integral weight twisted modular form on~$\Delta(2,5,\infty)$ from solutions of systems of nonlinear ODEs.
\end{abstract}

\section{Introduction}\label{section introduction}
Several papers in number theory, geometry, and mathematical physics deal with algebraic systems of nonlinear first order ODEs. A recurring example is the classical Ramanujan system
\begin{equation} \label{eq ramanujan}
 \left \{
\begin{aligned}
P'&\=P^2-\frac{Q}{144}\\
Q'&\=4PQ-\frac{R}{3} \\
R'&\=6PR-\frac{Q^2}{2}
\end{aligned}
\right.,
\end{equation}
whose solutions are the Eisenstein series $P=E_2/12,Q=E_4,R=E_6$ that generate the space of quasi-modular forms $\qmfs(\SL2)$. Here $':=q\tfrac{\partial}{\partial q}$ and~$q=e^{2\pi i\tau},\tau\in\Po$.
Similar systems were first considered by Darboux~\cite{da78} and Halphén~\cite{ha81}. The Ramanujan system~\eqref{eq ramanujan} plays a role in Nesterenko's proof~\cite{Nesterenko} of the algebraic independence over~$\Q$ of at least three numbers among~$q,E_2(q),E_4(q),E_6(q)$ for any~$q\in\C$ with~$0<|q|<1$ (see also Zudilin~\cite{zud03}).
Instances of systems of nonlinear ODEs related to mirror symmetry appear in the works of Alim et al.~\cite{ali17,murmar}.

From a geometric perspective, these systems of differential equations are often related to elliptic curves or K3-surfaces. Movasati \cite{ho14} proved that the Ramanujan system~\eqref{eq ramanujan} is the unique vector field on the moduli space of the family of the elliptic curves
\begin{equation} \label{eq Weierstrass family}
 y^2=4(x-t_1)^3-t_2(x-t_1)-t_3\,,\quad(t_1,t_2,t_3)\in\C^3\;\text{with }\,27t_3^2-t_2^3\neq 0\ ,
\end{equation}
that satisfies a certain equation involving the Gauss-Manin connection of the universal family of \eqref{eq Weierstrass family}.
Pursuing this interpretation, Movasati~\cite{ho22},\cite{GMCD-MQCY3} introduced a new technique called~\emph{Gauss-Manin connection in disguise} and, together with the second author~\cite{movnik}, used it to associate a canonical system of nonlinear ODEs to a moduli space of enhanced Calabi-Yau $n$-folds arising from the Dwork family, for any $n\in \N$. An interesting arithmetic aspect of their work is the following. The solutions of these systems for $n=1,2$ are quasimodular forms on congruence subgroups; for~$n=3,4$, the~$q$-expansion solutions have integral coefficients but, due to their fast growth, cannot be related to classical quasimodular forms.
The functions generated by the solutions of these systems were called~\emph{Calabi-Yau (CY) modular forms}.
In subsequent works \cite{younes2,younes3}, the second author deepened the analogy between CY modular forms and quasimodular forms; in particular, he showed the existence of a natural~\emph{Rankin-Cohen structure} on the space of CY modular forms.

Given a Fuchsian group~$\Gamma$, the~\emph{Rankin-Cohen brackets}~$[\,,]_n$ on the space of modular forms~$\mfs(\Gamma)$ are defined, for~$f\in\mfs_k(\Gamma)$ and~$g\in\mfs_l(\Gamma)$, by
\[
  [f,g]_n:=\sum_{r+s=n}(-1)^j\frac{(k+n-1)_s}{s!}\frac{(l+n-1)_r}{r!}f^{(r)}g^{(s)} \,\in\mfs_{k+l+2n}(\Gamma)\,,
\]
where~$f^{(r)}=d^rf/d\tau^r$ and~$(\alpha)_s=\alpha(\alpha+1)\cdots(\alpha+s-1)$.
More generally, a~\emph{Rankin-Cohen (RC) structure} is an associative commutative graded algebra together with a (countable) collection of bilinear operators that satisfy all the algebraic relations satisfied by the Rankin-Cohen brackets (see Section~\ref{sec:rc} for the definitions and properties of RC structures).
Other than in the theory of modular forms, RC brackets appear for instance in the study of pseuodifferential operators~\cite{CohenManinZagier}. Connes and Moscovici~\cite{ConnesMoscovici} generalized RC-structures by considering associative algebras endowed with an action of a certain Hopf algebra. Finally, El Gradechi~\cite{Elgradechi} recovered the classical Rankin-Cohen brackets while studying a general Lie-theoretic characterization of all the~$\mathrm{SL}_2(\R)$-equivariant holomorphic bi-differential operators on the upper half-plane.

Our first aim (Section~\ref{section preliminaries}) is to make the relation between systems of nonlinear ODEs and Rankin-Cohen structures explicit. The main result is the following (Theorems~\ref{thm 1} and~\ref{thm 2}).
\begin{theorem*}
Let~$D$ be a derivation on a graded algebra over a field $K$ of characteristic zero. Consider the following system of nonlinear ODEs
\begin{equation}
\label{eqn:rcintro}
\left\{
  \begin{array}{ll}
    D{\vs_1}=\vs_1^2+\prs_1(\vs_2,\vs_3,\ldots,\vs_\nvs)  \\\\
    D{\vs_2}={\wss_2}\vs_1\vs_2+\prs_2(\vs_2,\vs_3,\ldots,\vs_\nvs)  \\\\
    D{\vs_3}={\wss_3}\vs_1\vs_3+\prs_3(\vs_2,\vs_3,\ldots,\vs_\nvs)  \\
    \vdots  \\
    D{\vs_\nvs}={\wss_\nvs}\vs_1\vs_\nvs+\prs_\nvs(\vs_2,\vs_3,\ldots,\vs_\nvs)
  \end{array}
\right.\,,
\end{equation}
where~$\vs_j$ is of degree $\wss_j\in\Q\,,\wss_1=2$, and~$\prs_j(\vs_2,\vs_3,\ldots,\vs_\nvs)\in K[\vs_2,\dots,\vs_d]$ is a quasi-homogeneous polynomial of degree~$\wss_j+2$. The following equivalent statements hold:
\begin{enumerate}
\item The algebra~$M:=K[\vs_2,\dots,\vs_d]$ has a canonical RC structure. Conversely, every finitely generated canonical RC algebra arises in this way.
\item The finitely generated algebra~$\widetilde{M}=M[t_1]$ has a standard RC structure and is endowed with an~$\sl2$-module structure.
\end{enumerate}
\end{theorem*}
We call systems of the form~\eqref{eqn:rcintro}~\emph{Ramanujan systems of Rankin-Cohen type} or simply~\emph{RRC systems}.
\begin{rem} \label{rem alg dep}
\red{It may happen that the elements $t_1,t_2,\ldots,t_d$ are not algebraically independent over $K$. In this case, in order to have a free graded algebra, here and in the rest of this manuscript, we can substitute $K[t_1,t_2,\ldots,t_d]$ by
    \[
    \hat{K}[t_1,t_2,\ldots,t_d]:=\frac{K[t_1,t_2,\ldots,t_d]}{\mathcal{I}},
    \]
    and the elements $t_j$ by $\hat{t}_j:=t_j+\mathcal{I}$, where $\mathcal{I}$ is the sub-algebra generated by all polynomial relations $P(t_1,t_2,\ldots,t_d)=0$ over $K$. Under these changes, all the facts and the proofs stay the same.}
\end{rem}

It is natural to ask which RC structures arise from this construction in familiar situations.
As a starting point, in the second part of the paper (Section~\ref{sec:notation}), we associate to every hyperbolic triangle a system of nonlinear ODEs, via the uniformizing hypergeometric differential equation and a result of Ohyama~\cite{ohy96}.
These systems generalize the classical Ramanujan system~\eqref{eq ramanujan}: the solutions are algebraically independent~\emph{twisted modular forms} in the sense of Möller-Zagier~\cite{MoellerZagier} (see below in the introduction and  Section~\ref{sec:tmf} for the definition).
The main tool in the construction is the  modular embedding attached to every triangle group~$\Delta(n,m,\infty)$. Any such group has an inclusion map~$\iota\colon\Delta(n,m,\infty)\hookrightarrow G$ into a group~$G$ acting properly discontinuously on a product of upper half-planes~$\Po^h$. A modular embedding for~$\Delta(n,m,\infty)$ is a holomorphic map~$\phi=(\phi_j)_{j=1}^h\colon\Po\to\Po^h$ such that
\[
\phi(\gamma\cdot\tau)\=\iota(\gamma)\cdot\phi(\tau)\,,\quad\text{for every }\gamma\in\Delta(n,m,\infty)\,.
\]

Each component~$\phi_j$ of the modular embedding~$\phi$ gives rise to an automorphy factor on~$\Delta(n,m,\infty)$
\[
J_{\phi_j}(\gamma,\tau):=c^{\sigma_j}\phi_j(\tau)+d^{\sigma_j}\,,\quad\gamma=\begin{pmatrix}a & b\\c&d\end{pmatrix}\in\Delta(n,m,\infty)\,,
\]
where~$\sigma_j\colon K_{n,m}\to\R$ are the embeddings of the totally real trace field~$K_{n,m}$ of~$\Delta(n,m,\infty)$ and~$x^{\sigma_j}:=\sigma_j(x)$.
It follows that, for every~$\vec{w}=(w_1,\dots,w_h)\in\Q^h$, the expression
\[
\bigl(f\bigl|_{\vec{w}}\gamma\bigr)(\tau)\:=f(\gamma\tau)\prod_{j=1}^h{J_{\phi_j}(\gamma,\tau)^{-w_j}}\,
\]
defines an action on the space of holomorphic functions on~$\Po$.
A holomorphic function~$f\colon\Po\to\C$ of moderate growth at~$\infty$ is a~\emph{twisted modular form} of weight~$\vec{w}$ on~$\Delta(n,m,\infty)$ with respect to~$\phi$ if~$f|_{\vec{w}}\gamma=f$ for every~$\gamma\in\Delta(n,m,\infty)$. A twisted modular form of weight~$(w_1,0,\dots,0)$ is a classical modular form of weight~$w_1$,

Our result is the following (for a more precise formulation see Theorem~\ref{thm:hymain} and Corollary~\ref{cor:dim}).
\begin{theorem*}
Let~$n\le m\in\Z_{>0}$ be such that~$\tfrac{1}{n}+\tfrac{1}{m}<1$ and consider a modular embedding~$\phi=(\phi_j)_{j=1}^h\colon\Po\to\Po^h$ for~$\Delta(n,m,\infty)$.
\begin{enumerate}[wide=0pt]
\item For every~$j=1,\dots,h$ there exist~$k_j,r_j\in\Z_{>0}$ with~$\tfrac{k_j}{n}+\tfrac{r_j}{m}<1$  such that the system of nonlinear ODEs
\[
\left\{
\begin{aligned}
\frac{P_j'}{\phi_j'}&\=P_j^2\-\Bigl(\frac{mn-mk_j-nr_j}{2nm}\Bigr)^2Q_j^{m-2r_j}R_j^{n-2k_j}\\
\frac{Q_j'}{\phi_j'}&\=\frac{2n}{mn-mk_j-nr_j}P_jQ_j\-\frac{R_j^{n-k_j}Q_j^{1-r_j}}{m}\\
\frac{R_j'}{\phi_j'}&\=\frac{2m}{mn-mk_j-nr_j}P_jR_j\-\frac{Q_j^{m-r_j}R_j^{1-k_j}}{n}\,.
\end{aligned}
\right.
\]
admits algebraically independent solutions~$P_j,Q_j,R_j\colon\Po\to\C$ that are holomorphic on~$\Po$ and of moderate growth at~$\infty$.
\item The functions~$Q_j,R_j$ are twisted modular forms of rational weight on~$\Delta(n,m,\infty)$. More precisely, the polynomial algebra~$\C[Q_j,R_j]$ is the space of all twisted modular forms on~$\Delta(n,m,\infty)$ of weight~$(0,\dots,0,w_j,0,\dots,0)\,,w_j\in\Q$.
\end{enumerate}
\end{theorem*}

If~$(n,m)=(2,3)$, then~$h=1$ and we recover the Ramanujan system~\eqref{eq ramanujan}.
Nevertheless, other choices of~$(n,m)$ lead to entirely new  interesting systems (see the case of the arithmetic group~$\Delta(3,3,\infty)$ and of the non-arithmetic group~$\Delta(2,5,\infty)$ discussed in Section~\ref{sec:examples}.)

A number of corollaries follows from the above result. In Corollary~\ref{cor:dim} different modular characterizations of the polynomial ring~$\C[Q_j,R_j]$ and dimension formulae are given.

Corollaries~\ref{cor:hgde} and~\ref{cor:hgde'} concern the Gauss hypergeometric function~$F(\alpha,\beta;1;z)$. Let~$\phi=(\phi_j)\colon\Po\to\Po^h$ be a modular embedding for~$\Delta(n,m,\infty)$. For every~$j=1,\dots,h$ let~$k_j,r_j$ be as in Theorem~\ref{thm:hymain}, and write~$Q=Q_1,R=R_1,$ and~$N_j:=mn-nr_j-mk_j$. Then
\[
F\biggl(\frac{N_j+2nr_j}{2mn},\frac{N_j}{2mn};1;\frac{Q(\tau)^m-R(\tau)^n}{Q(\tau)^m}\biggr)^2\=\frac{Q(\tau)^{N_j/n+r_j-1}R(\tau)^{k_j-1}}{\phi_j'(\tau)}\,.
\]
Again, this generalizes the classical relation for~$Q=E_4$ and~$R=E_6$. Notice that in the above identity we are varying, as~$j$ varies, the parameters~$\alpha,\beta$ of~$F(\alpha,\beta;1;z)$ and leaving fixed the expansion parameter~$z$. In this way, we get on the right-hand side the derivative of all the components of the modular embedding~$\phi$.

Finally, in Corollary~\ref{cor:RC} we describe which of the systems appearing in the above theorem are of RRC type. We show in particular that in the case~$j=1$ the polynomial algebra~$\C[Q_1,R_1]$, which in general contains modular forms with different multiplier systems, always has a canonical Rankin-Cohen structure. In Section~\ref{sec:examples} we discuss in detail the case of the non-arithmetic group~$\Delta(2,5,\infty)$, for which Theorem~\ref{thm:hymain} gives two systems of ODEs. These, together with the restriction of Hilbert modular forms attached to~$\Q(\sqrt{5})$, permit us to construct the full space of twisted modular forms on~$\Delta(2,5,\infty)$, and attach to it a new system of ODEs describing the action of~$d/d\tau$.

\section{Rankin-Cohen structures and systems of nonlinear ODEs}\label{section preliminaries}

\subsection{Rankin-Cohen brackets and Rankin-Cohen structures}
\label{sec:rc}
In this section we recall some facts and terminologies from~\cite{zag94}.  Given a Fuchsian group of the first kind~$\Gamma$, we denote the graded algebra of modular forms and quasimodular forms on~$\Gamma$ respectively by~$\mfs(\Gamma)=\bigoplus_{k=0}^\infty{\mfs_k(\Gamma)}$ and~$\qmfs(\Gamma)=\bigoplus_{k=0}^\infty{\qmfs_k(\Gamma)}$ .


Let~$f\in\mfs_{k}(\Gamma)$ and~$g\in\mfs_{l}(\Gamma)$. Cohen~\cite{coh77} proved that for every~$n\in\N$ the bracket
\begin{equation}\label{eq rcb mf}
[f,g]_n:=\sum_{r+s=n}(-1)^j\frac{(k+n-1)_s}{s!}\frac{(l+n-1)_r}{r!}f^{(r)}g^{(s)} \,,
\end{equation}
where~$f^{(r)}=d^rf/d\tau^r$ and~$(\alpha)_s=\alpha(\alpha+1)\cdots(\alpha+s-1)$,
is a modular form of weight~$k+l+2n$. The bilinear operators~$[\,,]_n$, called~\emph{Rankin-Cohen brackets}, satisfy a number of algebraic identities, e.g.,
\begin{align*}
[f,g]_n-(-1)^n[g,f]_n&\=0\,,\\
[[f,g]_1,h]_1+[[g,h]_1,f]_1+[[h,f]_1,g]_1&\=0\,,\\
[[g,h]_0,f]_2-[[h,f]_0,g]_2+[[g,h]_2,f]_0-[[h,f]_2,g]_0&\=[[f,g]_1,h]_1\,.
\end{align*}
The second relation above shows in particular that~$(\mfs(\Gamma),[\,,]_1)$ is a Lie algebra. More identities can be found in~\cite{zag94}.
Let~$K$ be a field of characteristic zero. Abstracting from the modular case, we call a~\emph{Rankin-Cohen structure} (RC structure) any graded~$K$-vector space~$M=\bigoplus_{k\ge0}{M_{k}}$, ($\dim M_k<\infty$ for every~$k\ge0$) equipped with bilinear operations~$[\,,]_n\colon M_k\otimes M_l\to M_{k+l+2n}$ that satisfy all the algebraic identities satisfied by the Rankin-Cohen brackets.

\begin{rem}
\label{rem:RC}{\rm
In the following, we will consider RC algebras with grading~$k/N\,,k\in\Z_{\ge0}$, for a fixed positive integer~$N$. As will be clear, this causes no problems with the usual definitions and results.
}
\end{rem}

Given a commutative graded algebra~$M$ together with a derivation of degree two, i.e. a map~$D\colon M\to M$ with~$D(M_k)\subset M_{k+2}$ that satisfy Leibniz's rule, one can construct an RC algebra out of it: simply define the bilinear operators~$[\,,]_{D,n}$ on~$M$ as in~\eqref{eq rcb mf} where~$f^{(j)}$ has to be interpreted as~$D^j{f}$ (the same for~$g$). Such an RC algebra is called~\emph{standard RC algebra}. The prototypical example is the RC algebra of quasimodular form~$\qmfs(\Gamma)$ for every Fuchsian group~$\Gamma$, where~$D=d/d\tau$.

The following proposition of Zagier (Proposition 1 in~\cite{zag94}) gives yet another way to construct an RC algebra from a commutative associative algebra with a derivation (see also the version of this result discussed in~\cite{younes4}).
\begin{prop}
\label{prop crcab}
Let~$M$ be a commutative and associative graded~$K$-algebra with a derivation~$\partial\colon M_*\to M_{*+2}$ of degree~$2$, and let~$\Phi\in M_4$. Define brackets~$[,]_{\partial,\Phi,n}$ on~$M$ by
\begin{equation}
\label{eqn:rccan}
[f,g]_{\partial,\Phi,n}\:=\sum_{r+s=n}{(-1)^r\frac{(k+n-1)_s}{s!}\frac{(l+n-1)_r}{r!}f_rg_s}\,\in M_{k+l+2n}
\end{equation}
where~$f_r\in M_{k+2r}, g_s\in M_{l+2s} (r,s\ge0)$ are defined recursively by
\[
f_{r+1}=\partial f_r+r(r+k-1)\Phi f_{r-1}\,,\quad g_{s+1}=\partial g_s+s(s+k-1)\Phi g_{s-1}
\]
with initial conditions~$f_0=f,g_0=g$. Then~$(M,[\,,]_{\partial,\Phi,n})$ is an~$RC$ algebra.
\end{prop}
It is worth mentioning how the proof of this result works. One considers an element~$\phi\notin M$ of degree~$2$ and an embedding~$M\subset\widetilde{M}:=M[\phi]$. Equipped with the derivation
\begin{equation}
\label{eqn:extD}
D(f)\:=\partial(f)+k\phi f\in \widetilde{M}_{k+2}\,,\quad D(\phi):=\Phi+\phi^2\in \widetilde{M}_4
\end{equation}
the space~$\widetilde{M}$ is a standard RC algebra, and the restriction of its brackets to~$M$ gives precisely the ones defined in~\eqref{eqn:rccan}.
An RC structure like the one described in Proposition~\ref{prop crcab} is called~\emph{canonical RC structure}, and the derivation~$\partial$ is called a~\emph{Ramanujan-Serre derivation}. The proof of Proposition~\ref{prop crcab} shows in particular that every canonical RC algebra embeds in a standard RC algebra.
The main example of a canonical~RC structure is the RC structure on the space of modular forms~$\mfs(\Gamma)$ for any non-compact Fuchsian group~$\Gamma$. In particular, when~$\Gamma=\SL2$, we have, for~$f\in M_k(\SL2)$,
\[
\rsdo f=f'-\frac{k}{12}E_2f=-\frac{E_6}{3}\frac{\partial f}{\partial E_4}-\frac{E_4^2}{2}\frac{\partial f}{\partial E_6}\,
\]
and~$\Phi=-E_4/144,\phi=E_2/12$.

We close this section with an observation on the standard RC algebra~$\qmfs(\Gamma)=\mfs(\Gamma)[\phi]$. Besides~$D=d/d\tau$, we have two further derivations: the multiplication by the weight~$W$, i.e. $Wf=kf$ if~$f\in\qmfs_k(\Gamma)$, and~$\delta=\partial/\partial\phi$. They define an $\sl2$-module structure on~$\qmfs(\Gamma)$, i.e.,
\[
[\cvf,D]=\rvf \ , \ \ [\rvf,D]=2D \ , \ \ [\rvf,\cvf]=-2\cvf
\]
(see~\cite[Part I, Chapter 5]{bhgz}).
\subsection{Ramanujan systems of Rankin-Cohen type}\label{section ??}

\begin{theo} \label{thm 1}
Let~$D$ be a derivation on a graded algebra over a field~$K$ of characteristic zero. Consider the following system of nonlinar ODEs
\begin{equation} \label{eq mods}
\left\{
  \begin{array}{ll}
    D{\vs_1}=\vs_1^2+\prs_1(\vs_2,\vs_3,\ldots,\vs_\nvs)  \\\\
    D{\vs_2}={\wss_2}\vs_1\vs_2+\prs_2(\vs_2,\vs_3,\ldots,\vs_\nvs)  \\\\
    D{\vs_3}={\wss_3}\vs_1\vs_3+\prs_3(\vs_2,\vs_3,\ldots,\vs_\nvs)  \\
    \vdots  \\
    D{\vs_\nvs}={\wss_\nvs}\vs_1\vs_\nvs+\prs_\nvs(\vs_2,\vs_3,\ldots,\vs_\nvs)
  \end{array}
\right.\,,
\end{equation}
where~$\vs_j$ is of degree $\wss_j\in\Q\,,\wss_1=2$, and~$\prs_j(\vs_2,\vs_3,\ldots,\vs_\nvs)\in K[\vs_2,\dots,\vs_d]$ is a quasi-homogeneous polynomial of degree~$\wss_j+2$.
\begin{enumerate}
\item The algebra~$K[\vs_1,\dots,\vs_d]$ is a standard RC algebra with derivation~$D$, and the sub-algebra~$K[\vs_2,\dots,\vs_d]$ is a canonical RC algebra.
\item Conversely, every finitely generated canonical RC algebra arises in this way.
\end{enumerate}
\end{theo}
\begin{proof}
\begin{enumerate}[left =0pt]
\item To check that~$\widetilde{M}:=K[t_1,\dots,t_d]$ is a standard RC algebra, it is enough to notice that~$D$ is a degree two derivation on~$\widetilde{M}$ and apply the construction described after Remark~\ref{rem:RC}.

The algebra~$M$ generated over~$K$ by~$\vs_2,\dots,\vs_d$ is naturally a commutative and associative graded algebra. It has a derivation~$\partial\colon M\to M$ of weight two defined on the generators by
\[
\partial\vs_j\:=D\vs_j-\wss_j\vs_j\vs_1\=\prs_j(\vs_2,\vs_3,\ldots,\vs_\nvs)\in M_{\wss_j+2}\,,\quad j\ge 2\,.
\]
Set~$\Phi:=\prs_1(\vs_2,\dots,\vs_d)$. Proposition~\ref{prop crcab} applies and~$(M,[\,,]_{\partial,\Phi,n})$ is a canonical RC structure.

\item Let $(M,[\cdot , \cdot]_n)$ be a canonical RC algebra with derivation~$\partial$ and let~$\Phi\in M_4$. Assume that~$M$ is generated over~$K$ by~$\vs_2,\dots,\vs_d$. As explained after Proposition~\ref{prop crcab}, we can embed~$M$ into a standard RC algebra~$\widetilde{M}=M[\phi]$, for some degree-two element~$\phi\notin M$, with derivation~$D$ given in~\eqref{eqn:extD}. Then if we define~$\prs_j(\vs_2,\dots,\vs_d):=\partial\vs_j\in M_{\wss_j+2},\,j\ge2$ we have
\[
D\vs_j\=\wss_j\vs_j\phi+\prs_j(\vs_2,\dots,\vs_d)\,,\quad j\ge2\,.
\]
Setting~$\prs_1(\vs_2,\dots,\vs_d):=\Phi$, we get, again from~\eqref{eqn:extD}, that~$D\phi=\phi^2+\prs_1(\vs_2,\dots\vs_d)$. It finally follows that~$M$ is generated by the solutions of system~\eqref{eq mods} with~$\vs_1:=\phi$.
\end{enumerate}
\red{For the case that $t_1,t_2,\ldots,t_d$ are not algebraically independent over $K$ see Remark \ref{rem alg dep}.}
\end{proof}

We call the systems of the form~\eqref{eq mods}~\emph{Ramanujan systems of Rankin-Cohen type} (RRC~system for short).
The next result gives another characterization of {\mods}: they are precisely the systems whose solution space can be endowed with an $\sl2$-module structure.
Before stating the theorem recall that we can denote a differential operator $R$ on $\widetilde{M}=M[\vs_1]$ by $R=\sum_{i=1}^{\nvs} R_i \frac{\partial}{\partial \vs_i}$, where $R_i=R \, \vs_i,\ i=1,2,\ldots,\nvs$. If $S=\sum_{i=1}^{\nvs} S_i \frac{\partial}{\partial \vs_i}$ is another differential operator, then the Lie bracket $[R,S]$ is given as follows:
\begin{equation}\label{eq Lie brack}
  [R,S]=R S-S R=\sum_{i=1}^{\nvs}\big( R(S_i)-S(R_i)\big) \frac{\partial}{\partial \vs_i}.
\end{equation}

\begin{theo} \label{thm 2}
A finitely generated RC algebra $(M,[\,,]_n)$ is a canonical RC algebra if and only if it is a sub-RC algebra of a standard RC algebra $(\widetilde{M}:=M[\vs_1],[\cdot,\cdot]_{D,\ast})$ with the following property: the derivation $D$ along with the weight operator $\rvf$ ($\rvf f=k f\,,\,f\in \widetilde{M}_k$) and a derivation~$\cvf$ satisfying~$\ker  \cvf=M$, endows $\widetilde{M}$ with an $\sl2$-module structure.
\end{theo}
\begin{proof}
We first suppose that $(M,[\,,]_n)$ is a canonical RC algebra. As in Theorem \ref{thm 1}, $(M,[\,,]_n)$ is a sub-RC algebra of a standard RC algebra~$(\widetilde{M}=M[\vs_1],[\,,]_{D,n})$ with derivation~$D=\sum_{j=1}^{\nvs}D{\vs_j}\frac{\partial}{\partial \vs_j}$, where $\vs_1,\dots,\vs_d$ are as in system  \eqref{eq mods}.
If we set~$\rvf=\sum_{j=1}^{\nvs} \wss_j\vs_j\frac{\partial}{\partial \vs_j}$ and~$\cvf=-\frac{\partial}{\partial \vs_1}$, by direct computations one checks that~$[D,\cvf]=\rvf$ and $[\rvf,\cvf]=-2\cvf$. Finally, the computation
\[
[\rvf,D](f)=\rvf\big(D(f)\big)-D\big(\rvf(f)\big)=(k+2)D(f)-kD(f)=2D(f)\,,\quad f\in\widetilde{M}_k\,,
\]
proves that~$[\rvf,D]=2D$.

In order to prove the sufficient condition of the theorem, let $\{\vs_2,\vs_3,\ldots,\vs_\nvs\}$ be a set of generators of $M$ with $\vs_j\in M_{\wss_j}$, and set $\wss_1:=2$. The weight operator is then of the form~$\rvf=\sum_{j=1}^{\nvs} \wss_j\vs_j\frac{\partial}{\partial \vs_j}$. The relation $[\rvf,\cvf]=-2\cvf$ implies that $\cvf$ decreases the degree of any homogeneous element of $\widetilde{M}$ by two; hence it is of the form~$\cvf=\sum_{j=1}^{\nvs} \cvf_j\frac{\partial}{\partial \vs_j}$ where~$\cvf_j=\cvf(\vs_j)\in \widetilde{M}_{\wss_j-2}, \ j=1,2,\ldots,\nvs$. On the other hand, the hypothesis~$\ker \cvf=M$ implies that~$\cvf_j=0, \ j=2,3,\ldots,\nvs$, and therefore~$\cvf=\cvf_1\frac{\partial}{\partial \vs_1}, \ \cvf_1\in \widetilde{M}_0$. If we write $D=\sum_{j=1}^{\nvs} D_j\frac{\partial}{\partial \vs_j}$ where $D_j=D \vs_j\in \widetilde{M}_{\wss_j+2}$, $j=1,2,\ldots,\nvs$, then $[D,\cvf]=\rvf$ yields:
\begin{equation}\label{eq [D,cvf]=rvf j}
D(\cvf_1)-\cvf_1\frac{\partial D_1}{\partial \vs_1}=2\vs_1\,,\quad -\cvf_1\frac{\partial D_j}{\partial \vs_1}=\wss_j\vs_j, \ \ j=2,3,\ldots, \nvs.
\end{equation}
We consider the following two cases.
\begin{description}[left=0pt]
  \item[Case 1]
If $M$ does not contain any element of non-zero degree, i.e., $\wss_j=0$ for all $2\leq j \leq \nvs$, then \eqref{eq [D,cvf]=rvf j} implies that $\frac{\partial D_j}{\partial \vs_1}=0, \ j=2,3,\ldots,\nvs$. Hence, $M$ is closed under $D$, from which it follows that $(M,[\,,])$ is an standard RC algebra, and consequently it is a canonical RC algebra with $\rsdo=D$ and $\ew4=0$.
\item[Case 2]
If there exists $2\leq j \leq \nvs$ such that $\wss_j\neq 0$, then \eqref{eq [D,cvf]=rvf j} implies that $\cvf_1$ must be a constant element, i,e., $\cvf_1\in K$; without loss of generality we can assume~$\cvf=-\frac{\partial}{\partial \vs_1}$.
Then~\eqref{eq [D,cvf]=rvf j} imply that~$D_1=\vs_1^2+\prs_1(\vs_2,\vs_3,\ldots,\vs_\nvs)$ and~$D_j=\wss_j\vs_1\vs_j+\prs_j(\vs_2,\vs_3,\ldots,\vs_\nvs)\,,j=2,\ldots,\nvs,$ for some~$\prs_j(\vs_2,\vs_3,\ldots,\vs_\nvs)\in M_{\wss_j+2}$. Therefore, $\vs_1,\dots,\vs_\nvs$ satisfy an RRC system \eqref{eq mods} and Theorem \ref{thm 1} implies that $(M,[\,, ]_n)$ is a canonical RC algebra.
\end{description}
\end{proof}

\begin{rem}
\label{rmk:Delta}{\rm
In certain cases, one can read the RC structure directly from system~\eqref{eq mods}. Let~$(M,[\,,]_n)$ be an RC algebra, and assume that there exists a homogeneous element~$F\in M_\wss$ that is not a zero divisor and such that~$[F,M]_1\subset M\cdot F$ and~$[F,F]_2\in M\cdot F^2$. Then if we set
\[
\Phi\:=\frac{[F,F]_2}{\wss^2(\wss+1)F^2}\in M_4\,,
\]
and define for every~$f\in M_k$
\[
\rsdo f:=\frac{[F,f]_1}{\wss F}\in{M}_{k+2}
\]
the brackets~$[\,,]_n$ of~$M$ can be realized as~$[\,,]_{\rsdo,\Phi,n}$ (see~\eqref{eqn:rccan} for the definition, and~\cite[\S 6]{zag94} for a proof of this fact). It follows in particular from the proof of Theorem~\ref{thm 1} that the RRC system~\eqref{eq mods} is of the form
\begin{equation} \label{eq mods2}
\left\{
  \begin{array}{ll}
    \dot{\vs_1}=\vs_1^2+ \frac{[\red{F,F}]_2}{\wss^2(\wss+1)F^2} \\\\
    \dot{\vs_2}={\wss_2}\vs_1\vs_2+\frac{[F,\vs_2]_1}{\wss F}  \\\\
    \dot{\vs_3}={\wss_3}\vs_1\vs_3+\frac{[F,\vs_3]_1}{\wss F}  \\
    \vdots  \\
    \dot{\vs_\nvs}={\wss_\nvs}\vs_1\vs_\nvs+\frac{[F,\vs_\nvs]_1}{\wss F}
  \end{array}
\right.\,.
\end{equation}
As explained in~\cite{zag94}, the knowledge of the~$0$-th bracket (i.e. multiplication in~$M$), of the first bracket of a fixed element~$F$ with all elements of~$M$, and of the bracket~$[F,F]_2$ is enough to determine all other brackets. It follows that the \red{whole} RC structure can be read from system~\eqref{eq mods2}. For instance, the classical Ramanujan system~\eqref{eq ramanujan} is of the above form for~$F=\Delta=(Q^3-R^2)/1728$. We have in fact~$\wss=12$ and
\[
[\Delta, Q]=-4R\Delta\,,\quad [\Delta, R]=-6R\Delta\,,\quad [\Delta,\Delta]_2=-13Q\Delta^2\,.
\]
We generalize this observation to triangle groups in Corollary~\ref{cor:RC}.

In the case one of the conditions  $[F,F]_2\in M\cdot F^2$ or $[F,M]_1\subset M\cdot F$ does not hold, we can substitute $M$ by $\widehat{M}:=M\left[\vs_{\nvs+1}\right]$, where $\vs_{\nvs+1}=\frac{1}{F}$, and set $\widetilde{M}:=\widehat{M}[\vs_1]$. It is evident that $\widehat{M}$ is generated by $\vs_2,\vs_3,\ldots,\vs_\nvs,\vs_{\nvs+1}$ and we find $\dot{\vs_{\nvs}}_{+1}=-w\vs_1\vs_{\nvs+1}$. Thus, in this case the {\mods} is as follows
\begin{equation} \label{eq mods4}
\left\{
  \begin{array}{ll}
    \dot{\vs_1}=\vs_1^2+ \frac{[\red{F,F}]_2}{\wss^2(\wss+1)}\vs_{\nvs+1}^2 \\\\
    \dot{\vs_2}={\wss_2}\vs_1\vs_2+\frac{[F,\vs_2]_1}{\wss}\vs_{\nvs+1}  \\\\
    \dot{\vs_3}={\wss_3}\vs_1\vs_3+\frac{[F,\vs_3]_1}{\wss}\vs_{\nvs+1}  \\
    \vdots  \\
    \dot{\vs_\nvs}={\wss_\nvs}\vs_1\vs_\nvs+\frac{[F,\vs_\nvs]_1}{\wss}\vs_{\nvs+1}\\\\
     \dot{\vs_{\nvs}}_{+1}={-\wss}\vs_1\vs_{\nvs+1}
  \end{array}
\right.\,.
\end{equation}
}
\end{rem}



\section{RRC systems attached to hyperbolic triangles}
\label{sec:notation}

In this section we attach RRC systems to certain triangle groups, or equivalently, to the hypergeometric differential equation (HGDE for short)
\begin{equation}
\label{eqn:hgde}
z(1-z)\frac{d^2y}{dz^2}\+(\gamma-(\alpha+\beta+1)z)\frac{dy}{dz}\-\alpha\beta y\=0
\end{equation}
for certain choices of the parameters~$\alpha,\beta,\gamma$.

The relation between HGDE and systems of nonlinear differential equations has been investigated in full generality by Ohyama~\cite{ohy96}. Consider the projectively equivalent form (the~\emph{$Q$-form}) of the hypergeometric differential equation~$\frac{dy^2}{dz^2}-\mathcal{Q}(z)y=0$, where
\begin{equation}
\label{eqn:Qform}
\mathcal{Q}(z):=\frac{a}{z^2}\+\frac{b}{(z-1)^2}\+\frac{c}{z(z-1)},
\end{equation}
and
\[
a:=\frac{\gamma(\gamma-2)}{4},\quad b:=\frac{(\alpha+\beta-\gamma)^2-1}{4},\quad c:=\frac{\gamma(\alpha+\beta-\gamma+1)-2\alpha\beta}{2}\,.
\]
If~$u$ is a solution of~\eqref{eqn:hgde}, a solution of~$\frac{dy^2}{dz^2}-\mathcal{Q}(z)y=0$ is given by
\begin{equation}
\label{eqn:y}
y\=z^{\gamma/2}(1-z)^{(\alpha+\beta-\gamma+1)/2}u\,.
\end{equation}
Let~$\tau$ be the ratio of linearly independent solutions of~$\frac{dy^2}{dz^2}-\mathcal{Q}(z)y=0$. In~\cite[Sec. 3]{ohy96} it is proven that the functions
\begin{equation}
\label{eqn:XYZ}
X\:=\frac{d}{d\tau}\log{y},\quad Y\:=\frac{d}{d\tau}\log{\frac{y}{z}},\quad Z:=\frac{d}{d\tau}\log{\frac{y}{z-1}}
\end{equation}
satisfy the following system of nonlinear ODEs
\begin{equation}
\label{eqn:ohyama}
\left\{
\begin{aligned}
\frac{dX}{d\tau}&=X^2+a(X-Y)^2+b(X-Z)^2+c(X-Y)(X-Z)\\
\frac{dY}{d\tau}&=Y^2+a(X-Y)^2+b(X-Z)^2+c(X-Y)(X-Z)\\
\frac{dZ}{d\tau}&=Z^2+a(X-Y)^2+b(X-Z)^2+c(X-Y)(X-Z)\,.
\end{aligned}
\right.
\end{equation}
The Lie algebra generated by the vector field representation of the Ohyama system~$\frac{dX}{d\tau}\frac{\partial }{\partial X}+\frac{dY}{d\tau}\frac{\partial }{\partial Y}+\frac{dZ}{d\tau}\frac{\partial }{\partial Z}$ along with
\[
\rvf\=2X\frac{\partial }{\partial X}+2 Y\frac{\partial }{\partial Y}+2 Z\frac{\partial }{\partial Z}\,,\quad\cvf\=-\frac{\partial }{\partial X}-\frac{\partial }{\partial Y}-\frac{\partial }{\partial Z}\,,
\]
is isomorphic to $\sl2$. Hence it follows from Theorem~\ref{thm 1} and Theorem~\ref{thm 2} that the Ohyama system can be transformed to an \mods. In the following we construct such transformation for choices of~$\alpha,\beta,\gamma\in\Q$ related to hyperbolic triangles.
There are several possibilities. For instance, the transformation
\[
(X,Y,Z)\;\mapsto\;\biggl(\frac{Y+Z}{2},\frac{X-Z}{2},\frac{X-Y}{2}\biggr)
\]
already brings system~\eqref{eqn:ohyama} in RRC form. As will be clear below (Remark~\ref{rem XYZ}), the solutions of this system are not necessarily holomorphic functions on the upper half-plane. We will instead find a transformation that gives an RRC system with holomorphic solutions.
To this end, we need to recall the mapping properties of the ratio of solutions of HGDE and their relation with (twisted) modularity.

\subsection{Triangle groups}
\label{sec:triangle}
In this section we consider hypergeometric differential equations~\eqref{eqn:hgde} with parameters
\begin{equation}
\label{eqn:hypar}
\alpha\=\frac{1}{2}\biggl(1+\frac{r}{m}-\frac{k}{n}\biggr),\quad\beta\=\frac{1}{2}\biggl(1-\frac{k}{n}-\frac{r}{m}\biggr)\quad\gamma=1\,,\qquad\frac{k}{n}+\frac{r}{m}<1\,,
\end{equation}
where~$m,n,k,r\in\Z_{>0}$. Hypergeometric differential equations with the above choice of parameters are related to conformal mappings of hyperbolic triangles (see~\cite[Chapter V]{Nehari} for the general theory). More precisely, the ratio of linearly independent solutions of~\eqref{eqn:hgde} with parameters~\eqref{eqn:hypar} maps conformally the upper half-plane to the interior of a hyperbolic triangle~$\mathcal{T}(\tfrac{k}{n},\tfrac{r}{m})$ with angles~$(0,\tfrac{k\pi}{n},\tfrac{r\pi}{m})$. The monodromy group of the differential equation is then conjugated to the reflection group~$\Delta(\tfrac{k}{n},\tfrac{r}{m})$  of the triangle~$\mathcal{T}(\tfrac{k}{n},\tfrac{r}{m})$. These groups, generated by the reflections across the sides of the triangle~$\mathcal{T}(\tfrac{k}{n},\tfrac{r}{m})$, are known as~\emph{hyperbolic triangle groups}.
If we fix a choice~$\{v_1:=e^{\pi i(m-r)/m},v_2:=e^{\pi ik/n},v_3:=i\infty\}$, for the vertices of~$\mathcal{T}(\tfrac{k}{n},\tfrac{r}{m})$, and set~$\lambda(\tfrac{k}{n},\tfrac{r}{m}):=2\cos(\tfrac{\pi r}{m})+2\cos(\frac{\pi k}{n})$ we can give explicitly the generators of the reflection group by
\begin{equation}
\label{eqn:genD}
S=S_n\:=\begin{pmatrix}
-2\cos\bigl(\tfrac{\pi k}{n}\bigr) & 1\\
-1 & 0
\end{pmatrix}\,,\quad
T=T_{n,m}\:=\begin{pmatrix}
1 & \lambda(\tfrac{k}{n},\tfrac{r}{m})\\
0 & 1
\end{pmatrix}\,.
\end{equation}
The properties of the groups~$\Delta(\tfrac{k}{n},\tfrac{r}{m})$ turn out to be dramatically different in the case~$(k,r)=(1,1)$ and in the case where~$k>1$ or~$r>1$.

\subsubsection{Schwarz triangle groups ($k=r=1$)}
\label{sec:schwarz}
Assume now that~$k=r=1$ in~\eqref{eqn:hypar}. In this case the monodromy group of the hypergeometric differential equations is known as~\emph{Schwarz triangle group} and traditionally denoted by~$\Delta(n,m,\infty)$. It has the following presentation
\begin{equation}
\label{eqn:prestri}
\Delta(n,m,\infty)\:=\Delta(\tfrac{1}{n},\tfrac{1}{m})\=\langle{S_n,T_{n,m} | S_n^n=(S_nT_{n,m})^m=1}\rangle\,,
\end{equation}
and is the reflection group of the hyperbolic triangle~$\mathcal{T}(n,m):=\mathcal{T}(\tfrac{1}{n},\tfrac{1}{m})$ of angles~$(0,\tfrac{\pi}{m},\tfrac{\pi}{n})$. The main peculiarity of the groups~$\Delta(n,m,\infty)$, as is well known, is that they are discrete subgroups of~$\mathrm{SL}_2(\R)$, i.e., Fuchsian groups. A fundamental domain for their action is provided by the union of a triangle~$\mathcal{T}(n,m)$ and its reflection along one of its sides. It follows that the Schwarz triangles~$\mathcal{T}(n,m)$ tessellate the upper half-plane~$\Po$.
This fact has the following important consequence. As mentioned above, a ratio of independent solutions of the hypergeometric differential equation~\eqref{eqn:hgde} with parameters as in~\eqref{eqn:hypar}~$(k=r=1)$ maps conformally the upper half-plane into the hyperbolic triangle~$\mathcal{T}(n,m)$. Let~$t\colon\mathcal{T}(n,m)\to\Po$ be the inverse holomorphic map.
By repeated use of the Schwarz reflection principle, one can extend~$t$ to a holomorphic map on~$\Po$. By construction~$t$ is defined on the quotient~$\Po/\Delta(n,m,\infty)$ and is a biholomorphism~$\Po/\Delta(n,m,\infty)\to\C\smallsetminus\{0,1\}$; such a map is called a~\emph{Hauptmodul} for the triangle group~$\Delta(n,m,\infty)$. Moreover, one can build a theory of modular forms over the group~$\Delta(n,m,\infty)$. This is known as Gauss-Schwarz theory.

\subsubsection{Modular embeddings ($k>1$ or $r>1$)}
Consider now the case~$\mathcal{T}(\tfrac{k}{n},\tfrac{r}{m})$ is an hyperbolic triangle where at least one among~$k$ and~$r$ is larger than~$1$. As is well known~\cite{Hecke}, its reflection group $\Delta(\tfrac{k}{n},\tfrac{r}{m})$ is not a discrete subgroup of~$\mathrm{SL}_2(\R)$. From the mapping point of view, this means that the triangle~$\mathcal{T}(\tfrac{k}{n},\tfrac{r}{m})$ does not tessellate the upper half plane. It is then not possible to get, through a straightforward application of Schwarz's reflection, a holomorphic map~$\Po\to\C$ from the ratio of solutions of the HGDE associated to~$\mathcal{T}(\tfrac{k}{n},\tfrac{r}{m})$. Moreover, the only automorphic forms on these groups are constants. To deal with this situation we exploit the existence of a special map associated to the group~$\Delta(n,m,\infty)$ called modular embedding.

\smallskip

Let~$\Delta(n,m,\infty)$ be as in~\eqref{eqn:prestri} and let~$K_{n,m}:=\Q(\cos(\pi/n),\cos(\pi/m))$ be the totally real field generated over~$\Q$ by the traces of elements of~$\Delta(n,m,\infty)$ (see~\cite{Takeuchi}). Let~$\gamma_1,\gamma_2,\gamma_3$ be a set of generators of~$\Delta(n,m,\infty)$ that satisfy the relation~$\gamma_1\gamma_2\gamma_3=1$. Then the~$K_{n,m}$-vector space generated by~$\{\gamma_0:=1,\gamma_1,\gamma_2,\gamma_3\}$ is a quaternion algebra~$\mathcal{B}_\Delta$ over~$K_{n,m}$ and, if~$\mathcal{O}_{n,m}$ denotes the ring of integers of~$K_{m,n}$, the~$\mathcal{O}_{n,m}$-lattice~$\mathcal{O}_\Delta:=\oplus_{i=0}^3\mathcal{O}_{n,m}\gamma_i$ is an order in~$\mathcal{B}_\Delta$. It follows that we can identify the group~$\Delta(n,m,\infty)$ with a subgroup of the group~$\mathcal{O}^\times_{\Delta,1}$ of units of~$\mathcal{O}_\Delta$ of reduced norm one.
On the other hand, the theory of quaternion algebras guarantees the existence of an isomorphism
\begin{equation}
\label{eqn:quat}
\iota\colon\mathcal{B}_\Delta\otimes_\Q\R\to M_2(\R)^h\oplus H^{d-r}\,,
\end{equation}
for some~$1\le h\le d$, where~$d:=[K_{n,m}:\Q]$ and~$H$ denotes the Hamiltonian algebra. Under the map~$\iota$ the group~$\mathcal{O}^\times_{\Delta,1}$ maps to a discrete subgroup of~$\mathrm{SL}_2(\R)^h$ acting discontinuously on~$\Po^h$. It follows that the triangle group~$\Delta(n,m,\infty)$ can be realized as a subgroup of a group acting discontinuously on~$\Po^h$. Since the map~$\iota$ extends the real embeddings of~$K_{n,m}$, the group~$\Delta(n,m,\infty)$ acts on~$\Po^h$ via Galois conjugation.

The above construction works more generally for Fuchsian groups with totally real trace field. The peculiarity of triangle groups consists in the existence of an holomorphic map~$\phi\colon\Po\to\Po^h$, called~\emph{modular embedding}, such that for every~$\gamma\in\Delta(n,m,\infty)$,
\begin{equation}
\label{eqn:modemb}
\phi\circ\gamma\=\iota(\gamma)\circ\phi\,.
\end{equation}
It is possible to normalize the map~$\phi=\bigl(\phi_1,\phi_2\dots,\phi_h\bigr)$ in such a way that~$\phi_1=\mathrm{Id}$; this is the normalization we adopt in the following. The map~$\phi$ was constructed for triangle groups in~\cite{CohenWolfart} in three different ways. Apart from the case of triangle groups (and arithmetic Fuchsian groups), examples of modular embeddings seem to be rare; an example related to Teichmüller curves  uniformized by non-arithmetic Fuchsian groups is discussed in~\cite{MoellerZagier}.

\subsubsection{Modular embedding of triangle groups and HGDE}
\label{sec:modembde}
Fix~$2\le n\le m$, and let~$k,r\in\Z_{>0}$ be such that~$k/n+r/m<1$.
The construction of the modular embedding~$(\phi_j)_{j=1}^h$ for~$\Delta(n,m,\infty)$ given in \cite{CohenWolfart} implies that there exists a unique~$j\in\{1,\dots,h\}$ such that the restriction of~$\phi_j$ to the open triangle $\mathcal{T}(n,m)$ is a conformal map onto~$\mathcal{T}(\tfrac{k}{n},\tfrac{r}{m})$.
For this reason, from now on we use a different notation: we denote~$k,r$ as above by~$k_j,r_j$ and let~$\mathcal{T}_j:=\mathcal{T}(\tfrac{k_j}{n},\tfrac{r_j}{m})$, so that~$\phi_j\mathcal{T}(n,m)=\mathcal{T}_j$. Moreover, we denote by~$y_j$ the holomorphic solution of the HGDE attached to the triangle~$\mathcal{T}_j$, and by~$\tau_j$ the ratio of two linearly independent solutions. From our normalization ($\phi_1=\mathrm{Id}$) we recover in particular that~$\tau_1$ extends to~$\tau\in\Po$.

The existence of the modular embedding for~$\Delta(n,m,\infty)$ lets us interpret the functions~$\tau_j$ and~$y_j$ as functions on the upper half-plane. We have by construction~$\tau_j=\phi_j(\tau)$ for~$\tau\in\mathcal{T}(n,m)$, and since~$y_j$ can be expressed as a function of~$\tau_j$, it becomes a function of~$\tau\in\mathcal{T}(n,m)$, by pre-composing it with~$\phi_{j}$. Finally, by repeated use of Schwarz's reflection, we can extend~$y_j$ to a function on~$\Po$.


\subsection{Twisted modular forms}
\label{sec:tmf}
A consequence of the non-discreteness of~$\Delta(\tfrac{k}{n},\tfrac{r}{m})$ is that the theory of automorphic forms on this group with respect to the standard automorphy factor of~$\mathrm{SL}_2(\R)$ reduces to constant functions. It makes sense then to consider new automorphy factors defined in terms of the modular embedding~$\phi$.

Fix~$2\le n\le m$ and consider a modular embedding~$\phi=(\phi_j)_{j=1}^h\colon\Po\to\Po^h$ associated to the group~$\Delta(n,m,\infty)$. Let~$K_{n,m}$ be the trace field of~$\Delta(n,m,\infty)$ and let~$\sigma_i\colon K_{n,m}\to\R$, for~$i=1,\dots,d=[K_{n,m}:\Q]$, be the real embeddings of~$K_{n,m}$ ($h\le d$, see~\eqref{eqn:quat}).

Let~$\vec{w}=(w_1,\dots,w_h)\in\Q^h$. An~\emph{automorphy factor of weight~$\vec{w}$ with respect to~$\phi$} is a map~$J_\phi\colon\Delta(n,m,\infty)\times\Po\to\C$, holomorphic on~$\Po$ and such that:
\begin{enumerate}
\item $J_\phi(\gamma_1\gamma_2,\tau)=J_\phi(\gamma_1,\gamma_2\tau)J_\phi(\gamma_2,\tau)$ for every~$\gamma_1,\gamma_2\in\Delta(n,m,\infty)$, (cocycle relation)
\item for every~$\gamma=\left(\begin{smallmatrix}a & b \\c & d\end{smallmatrix}\right)\in\Delta(n,m,\infty)$, if $x^{\sigma_j}$ denotes the image of~$x\in K_{n,m}$ via the real embedding~$\sigma_j$, we have
\[
\frac{d}{d\tau}\log J_\phi(\gamma,\tau)=\sum_{j=1}^h\frac{w_jc^{\sigma_j}}{c^{\sigma_j}\phi_j(\tau)+d^{\sigma_j}}\phi_j'(\tau)\,.
\]
\end{enumerate}
Examples are provided by the functions
\begin{equation}
\label{eqn:Jj}
J_{\phi_j}(\gamma,\tau)\:=c^{\sigma_j}\phi_j(\tau)+d^{\sigma_j}\quad\text{if }\gamma=\begin{pmatrix}a & b\\c & d\end{pmatrix}\in\Delta(n,m,\infty)\,,
\end{equation}
defined for~$j=1,\dots, h.$
An easy computation (see~\cite[Section 2]{MoellerZagier}) shows that~$J_{\phi_j}$ is an automorphy factor of weight~$(0,\dots0,w_j=1,0,\dots,0)$. For~$j=1$ and~$\phi_1=\mathrm{Id}$ we get the classical automorphy factor of~$\mathrm{SL}_2(\R)$.

Following~\cite{MoellerZagier} we define a~\emph{twisted modular form} associated to~$J_\phi$ to be a holomorphic function~$f\colon\Po\to\C$ that is bounded as~$\mathrm{Im}(z)\to\infty$ uniformly with respect to~$\mathrm{Re}(z)$, and that satisfies the modular transformation
\begin{equation}
\label{eqn:hymod}
f(\gamma\tau)\=J_\phi(\gamma,\tau)f(\tau),\quad\text{for every }\gamma\in\Delta(n,m,\infty)\,.
\end{equation}
Basic examples of twisted modular forms are:
\begin{enumerate}[wide=0pt]
\item Modular forms of weight~$k$ on~$\Delta(n,m,\infty)$ are twisted modular forms of weight~$(k,0,\dots,0)$ for any choice of~$\phi$;
\item When~$\Delta(n,m,\infty)$ embeds into a Hilbert modular group, the restriction of Hilbert modular forms of weight~$(w_1,\dots,w_h)$ are twisted modular forms of the same weight. Restrictions of Hilbert modular forms do not span in general the space of twisted modular forms.
\item The derivative~$\phi_j'$ of~$\phi_j$ is a twisted modular form of weight~$(2,0,\dots,0,w_j=-2,0,\dots,0)$. Explicitly, the transformation property of~$\phi$ in~\eqref{eqn:modemb} for the component~$\phi_j$ reads
\[
\phi_j(\gamma\tau)\=\frac{a^{\sigma_j}\phi_j(\tau)+b^{\sigma_j}}{c^{\sigma_j}\phi_j(\tau)+d^{\sigma_j}}\,,\quad\gamma=\begin{pmatrix}a & b \\c & d\end{pmatrix}\in\Delta(n,m,\infty)\,.
\]
Taking the derivative with respect to~$d/d\tau$ at both sides gives
\[
\phi_j'(\gamma\tau)\=\phi'_j(\tau)\frac{(c\tau+d)^2}{(c^{\sigma_j}\phi_j(\tau)+d^{\sigma_j})^2}\,.
\]
Moreover, the transformation property of~$\phi_j$ implies~$\phi(T\tau)=\phi_j\bigl(\tau+\lambda\bigr)=\phi_j(\tau)+\lambda^{\sigma_j}$ (where~$\lambda=\lambda\bigl(\frac{k_j}{n},\frac{r_j}{m}\bigr)$ in~\eqref{eqn:genD}). This implies that~$\phi_j(\tau)-\lambda^{\sigma_j}\tau/\lambda$ is invariant under~$\tau\to T\tau$. In particular,~$\phi_j(\tau)=\lambda^{\sigma_j}\tau/\lambda+g(q)$, where~$q=e^{2\pi i\tau/\lambda}$, for a function~$g$ holomorphic in the disk~$|q|<1$ punctured in~$q=0$. It can be shown that~$g$ extends to a holomorphic function in~$q=0$ (see Proposition 1.1 of~\cite{MoellerZagier} for details). This implies that~$\phi_j'$ is bounded as~$\tau\to\infty$.
\end{enumerate}

We denote by~$M_{J_\phi}(\Delta(n,m,\infty))$ the vector space of all twisted modular forms associated to~$J_\phi$. In the following, we will be interested in the space of all twisted modular forms of fixed weight
\[
M_{\vec{w}}(\Delta(n,m,\infty),\phi)\:=\bigoplus_{\overset{J_\phi}{\text{weight }\vec{w}}}M_{J_\phi}(\Delta(n,m,\infty))
\]
where the sum is over all automorphy factors~$J_\phi$ of weight~$\vec{w}$ with respect to~$\phi$ for~$\Delta(n,m,\infty)$.
A modular form with respect to an automorphy factor~$J_\phi$ of weight~$\vec{w}=(0,\dots,0,w_j=w,0,\dots,0)$ will be called of~\emph{pure $j$-th weight $w$}.

Let~$e_n,e_m$ denote the elliptic points of~$\Delta(n,m,\infty)$ of order~$n,m$ respectively in~$\pro\mathrm{SL}_2(\R)$. From now on we fix the following normalization for the Hauptmodul~$t$ of~$\Delta(n,m,\infty)$:
\begin{equation}
\label{eqn:divhaupt}
t(\infty)\=0\,,\quad t(e_n)=1\,,\quad t(e_m)=\infty\,.
\end{equation}
For future use, we collect simple information of the divisor of~$t'$ and~$\phi_j'$.
\begin{lemma}
\label{lem:divmod}
\begin{enumerate}[wide=0pt]
\item $\mathrm{div}(1-t)=e_n-e_m,\quad\mathrm{div}\,t'=\infty+\bigl(1-\frac{1}{n}\bigr)e_n+\bigl(-1-\frac{1}{m}\bigr)e_m\,.$
\item $\mathrm{div}\,\phi_j'=\frac{k_j-1}{n}e_n+\frac{r_j-1}{m}e_m\,.$
\end{enumerate}
\end{lemma}
\begin{proof}
Point (1) is immediate from the normalization~\eqref{eqn:divhaupt}.
Since~$\phi_j'$ is a twisted modular form (Example 3 in Section~\ref{sec:tmf}), the divisor of~$\phi_j'$ can be studied on a fundamental domain of~$\Delta(n,m,\infty)$. The expression for~$\mathrm{div}\,\phi_j'$ follows then from the construction of~$\phi$ in~\cite{CohenWolfart}.
\end{proof}

\begin{prop}
\label{prop:dimtws}
Let~$2\le n\le m \in\Z$ and let~$\phi=(\phi_j)\colon\Po\to\Po^h$ be a modular embedding for~$\Delta(n,m,\infty)$. Assume that~$\phi_j\mathcal{T}(n,m)=\mathcal{T}(\tfrac{k_j}{n},\tfrac{r_j}{m})$, and denote~$N_j:=mn-mk_j-nr_j$. For every~$\vec{w}=(w_1,\dots,w_h)\in\Q^h$:
\begin{enumerate}
\item If~$\sum_{j=1}^h{w_jN_j}\;\notin2\gcd(m,n)\Z$, then~$\dim M_{\vec{w}}(\Delta(n,m,\infty),\phi)=0$\,
\item If~$\sum_{j=1}^h{w_jN_j}\;\in2\gcd(m,n)\Z_{\ge0}$, then
\begin{equation}
\label{eqn:hydim}
\dim{M_{\vec{w}}(\Delta(n,m,\infty),\phi)}\le
\begin{cases}
1+\Bigl\lfloor{\sum_{j=1}^h\frac{w_jN_j}{2mn}}\Bigr\rfloor & \text{if } \Bigl\{\tfrac{\sum_{j=1}^h{w_jN_j}}{2mn}\Bigr\}=\tfrac{a}{n}+\frac{b}{m}				\\
\Bigl\lfloor{\sum_{j=1}^h\frac{w_jN_j}{2mn}}\Bigr\rfloor & \text{otherwise}\,,
\end{cases}
\end{equation}
where~$\{x\}\in [0,1)$ denotes the fractional part of~$x\in\Q$ and~$a,b\in\Z_{\ge0}$.
\end{enumerate}
\end{prop}
\begin{proof}
To prove the bound in the statement, we compute the degree of the divisor of an element~$f\in M_{\vec{w}}(\Delta(n,m,\infty),\phi)$. From the transformation property of~$f$~\eqref{eqn:hymod} it follows in fact that the zeros of~$f$ can be studied on a fundamental domain of~$\Delta(n,m,\infty). $
This standard computation is done by integrating~$d\log(f)$ on a compact subset of a fundamental domain of~$\Delta(n,m,\infty)$. A full proof in the case of~$\SL2$, i.e., $(n,m)=(2,3)$, is given in~\cite[Part I, Chapter 2]{bhgz}. We report here only the main ideas of the proof and refer to the above book for details (and a useful picture).

As a fundamental domain~$\mathcal{F}$ for~$\Delta(n,m,\infty)$ consider the union of the triangle of vertices~$\{\xi_n=e^{\pi i/n},\,\xi_m=e^{\pi i(m-1)/m},\,i\infty\}$ and its reflection across the vertical line~$\mathrm{Re}(\tau)=\mathrm{Re}(\xi_n)$.
Let~$f\in M_k(\Delta(n,m,\infty),\phi)$ and let~$D$ denote the domain obtained from~$\mathcal{F}$ by erasing small~$\epsilon$-neighborhoods of the zeros of~$f$, including the neighborhood at~$i\infty$, in such a way that all these neighborhoods are disjoint. Then~$f$ is non-zero on~$D$ and the integral of~$d\log(f)$ over~$\partial{D}$ is zero.
The contribution of each piece of the boundary~$\partial D$ is as follows:
\begin{itemize}[wide=0pt]
\item integration over the~$\epsilon$-neighborhoods: if~$\tau$ is not an elliptic point nor~$\infty$, the integral of~$d\log(f)$ over the boundary of its~$\epsilon$-neighborhood gives~$2\pi i\mathrm\cdot\mathrm{ord}_\tau(f)$ by Cauchy's theorem. Similarly, zeros at the elliptic points~$\xi_n,\xi_m$ contribute by~$2\pi i\cdot\mathrm{ord}_{\xi_n}(f)/n$ and~$2\pi i\cdot\mathrm{ord}_{\xi_m}(f)/m$ respectively, and the neighborhood at infinity gives~$2\pi i\cdot\mathrm{ord}_\infty(f)$.
\item integration over the lines: the vertical lines of the boundary of~$\mathcal{F}$ are identified by the transformation~$T$ and from~\eqref{eqn:hymod} and the definition of automorphy factor it follows that~$d\log f(T\tau)=d\log f(\tau)$. Then the sum of their contribution is zero, since the two lines have opposite orientation.
Integration over the arcs~$(\xi_m,\xi_n)$ and~$(\xi_n,\xi_{m}+\lambda)$ is slightly more involved, and the modular embedding appears.
These arcs are identified by the transformation~$S_n$, which fixes~$\xi_n$ and exchanges~$\xi_m$ and~$\xi_m+\lambda$. From the transformation property~\eqref{eqn:hymod} of~$f$ we see that
\[
\int_{\xi_m}^{\xi_n}{d\log f(\tau)}\+\int_{\xi_n}^{\xi_m+\lambda}{d\log f(\tau)}\=\int_{\xi_m}^{\xi_n}{d\log f(\tau)}-\int_{\xi_m}^{\xi_n}{d\log f(S\tau)}
\]
and by definition of~$J_\phi$ we have, for the modular embedding~$\phi=(\phi_j)_{j=1}^h$,
\[
d\log f(S\tau)\=d\log f(\tau)+d\log J_\phi(S,\tau)\=d \log f(\tau)\+\sum_{j=1}^h{w_j\cdot d\log\phi_j}\,.
\]
Consequently, the contribution of these arcs is given by the integral of~$-\sum_j{w_jd\log\phi_j}$ over the arc~$\xi_m$ to~$\xi_n$. It yields~$-\pi i(mn)^{-1}\sum_{j}{w_jN_j}$, since, by construction, we have~$\phi_j(\xi_n)=e^{\pi i k_j/n}$ and~$\phi_j(\xi_m)=e^{\pi i(m-r_j)/m}$.
\end{itemize}

Considering all the contributions together we get the generalized valence formula
\begin{equation}
\label{eqn:hydeg}
\sum_{\tau\in\Po/\Delta(n,m,\infty)}{\frac{\mathrm{ord}_\tau(f)}{n_\tau}}\+\mathrm{ord}_{\infty}(f)\=\sum_{j=1}^h\frac{w_jN_j}{2mn}\,,
\end{equation}
where~$n_\tau=n,m$ if~$\tau=\xi_n,\xi_m$ respectively, and~$n_\tau=1$ in all the other cases.
The valence formula immediately implies that~$M_{\vec{w}}(\Delta(n,m,\infty),\phi)=0$ if~$\sum_{j=1}^h{w_jN_j}\notin2\gcd(m,n)\Z$.
To prove the bounds in~\eqref{eqn:hydim} we proceed as follows. Let~$\sum_{j=1}^h{w_jN_j}\in2\gcd(m,n)\Z_{\ge0}$ and let~$s:=\lfloor{\tfrac{\sum{w_jN_j}}{2mn}}\rfloor+1$. Given~$s+1$ linearly independent modular forms of weight~$\vec{w}$, we can find a linear combination of those with a zero of order~$s$ at~$\infty$. Since~$s>\sum{w_jN_j}/(2mn)$ this contradicts~\eqref{eqn:hydeg}, and consequently we get the first bound in~\eqref{eqn:hydim}.
We can reduce this bound by one in the case there is no~$a,b\in\Z_{\ge_0}$ with~$\frac{\sum{w_jN_j}}{2mn}-\lfloor\frac{\sum{w_jN_j}}{2mn}\rfloor=\tfrac{a}{m}+\tfrac{b}{n}$. In this case, for a modular form~$f\in M_w(\Delta(n,m,\infty))$,  we necessarily have that~$\mathrm{ord}_{e_n}(f)/n+\mathrm{ord}_{e_m}(f)/m>1$. From the valence formula it follows that~$f$, outside the elliptic points, has at most~$\lfloor{\tfrac{\sum{w_jN_j}}{2mn}}\rfloor-1$ zeros. The same argument as before, now with~$s:=\lfloor{\tfrac{\sum{w_jN_j}}{2mn}}\rfloor$, proves the second bound in~\eqref{eqn:hydim}.
\end{proof}

\subsection{Hypergeometric systems of Rankin-Cohen type}

\begin{theo}
\label{thm:hymain}
\begin{enumerate}[wide=0pt]
\item
Let~$2\le n\le m$ be such that~$1/n+1/m<1$ and consider a modular embedding~$\phi=(\phi_j)_j\colon\Po\to\Po^h$ for~$\Delta(n,m,\infty)$ such that~$\phi_j\mathcal{T}(n,m)=\mathcal{T}\bigl(\tfrac{k_j}{n},\tfrac{r_j}{m}\bigr)$ and~$\phi_1=\mathrm{Id}$. Then for~$j=1,\dots,h$, the system of nonlinear ODEs
\begin{equation}
\label{eqn:hyrc}
\left\{
\begin{aligned}
\frac{P_j'}{\phi_j'}&\=P_j^2\-\Bigl(\frac{mn-mk_j-nr_j}{2nm}\Bigr)^2Q_j^{m-2r_j}R_j^{n-2k_j}\\
\frac{Q_j'}{\phi_j'}&\=\frac{2n}{mn-mk_j-nr_j}P_jQ_j\-\frac{R_j^{n-k_j}Q_j^{1-r_j}}{m}\\
\frac{R_j'}{\phi_j'}&\=\frac{2m}{mn-mk_j-nr_j}P_jR_j\-\frac{Q_j^{m-r_j}R_j^{1-k_j}}{n}\,.
\end{aligned}
\right.
\end{equation}
admits algebraically independent solutions~$P_j,Q_j,R_j\colon\Po\to\C$ that are holomorphic in~$\Po$ and of moderate growth at~$\infty$.
If~$t$ is a Hauptmodul for~$\Delta(n,m,\infty)$ normalized like in~\eqref{eqn:divhaupt}, then~$P_j,Q_j$, and~$R_j$ are given explicitly by
\begin{equation}
\label{eqn:hysol}
\begin{aligned}
P_j&\=\frac{t''\phi_j'-t'\phi_j''}{2t'\phi_j'^2}\-\frac{m+r_j}{2m}\frac{t'}{(t-1)\phi_j'}\+\frac{mk_j+nr_j}{2mn}\frac{t'}{t(t-1)\phi_j'}\,,\\
Q_j&\=\biggl(\frac{(t')^n}{t^n\phi_j'^n(1-t)^{n-k_j}}\biggr)^{1/(nm-nr_j-mk_j)}\,,\quad R_j\=\biggl(\frac{(t')^m}{t^m\phi_j'^m(1-t)^{r_j}}\biggr)^{1/(nm-nr_j-mk_j)}\,.
\end{aligned}
\end{equation}
\item The free polynomial algebra~$\C[Q_j,R_j]$ is the space of pure~$j$-th weight twisted modular forms on~$\Delta(n,m,\infty)$ with respect to~$\phi$
\begin{equation}
\label{eqn:freealg}
\C[Q_j,R_j]\=\bigoplus_{\vec{w}_j\in\Q^h}{M_{\vec{w}_j}\bigl(\Delta(n,m,\infty),\phi\bigr)}:=\bigoplus_{\vec{w}_j\in\Q^h}\bigoplus_{\overset{J_\phi}{\text{weight }\vec{w}_j}}M_{J_\phi}(\Delta(n,m,\infty))\,,
\end{equation}
where~$\vec{w}_j=(0,\dots,0,*,0,\dots,0)$ is non zero only in the~$j$-th position.

For~$j=1$, the ring of quasimodular forms~$\widetilde{M}_{*}(\Delta(n,m,\infty))$ is a subring of~$\C[P_1,Q_1,R_1]$.
\end{enumerate}
\end{theo}
\begin{proof}
\begin{enumerate}[wide=0pt]
\item
\red{Define~$\Delta_j:=Q_j^m-R_j^n$ and~$N_j=mn-mk_j-nr_j$.
Then~$P_j=\frac{N_j}{2mn}\frac{d\log(\Delta_j)}{d\phi_j}$ is the function given in~\eqref{eqn:hysol}. By Lemma~\ref{lem:divmod} it follows that
\[
\mathrm{div}\,Q_j=\frac{1}{m}e_m\,,\quad\mathrm{div}\,R_j=\frac{1}{n}e_n\,,\quad\mathrm{div}\,\Delta_j=\infty\,;
\]
in particular~$Q_j,R_j$ and~$P_j$ are holomorphic in~$\Po$. The behavior of~$P_j,Q_j,R_j$ at~$\infty$ can in fact be deduced from the explicit description~\eqref{eqn:hysol}.
It remains to show that~$P_j,Q_j$ and~$R_j$ satisfy system~\eqref{eqn:hyrc} (we prove the algebraic independence in the next point). Since~$\Delta_j=Q_j^mt$, it follows
\[
P_j\=\frac{N_j}{2mn}\frac{d\log(\Delta_j)}{d\phi_j}\=\frac{N_j}{2mn}\biggl(\frac{t'}{\phi_j't}+m\frac{Q_j'}{\phi_j'Q_j}\biggr)\,.
\]
This identity and~$\frac{t'}{t\phi_j'}=R_j^{n-k_j}Q_j^{-rj}$ imply the differential relation for~$Q_j'$ in~\eqref{eqn:hyrc}.
The relation for~$R_j'$ is proven similarly, after noticing that~$\Delta_j=R_j^nt/(1-t)$.
Finally, the expression~$\Delta_j=\bigl(\frac{dt}{d\phi_j}\bigl)^{\frac{mn}{N_j}}p_j(t)^{-1}$, where~$p_j(t)=t^{\frac{mk_j+nr_j}{N_j}}(1-t)^{\frac{m(n-k_j)}{N_j}}$, implies
\begin{equation}
\label{eqn:Pj}
\begin{aligned}
\frac{d P_j}{d\phi_j}-P_j^2&=\frac{N_j}{2mn}\frac{d^2\log(\Delta_j)}{d\phi_j^2}-\frac{N_j^2}{(2mn)^2}\biggl(\frac{d\log(\Delta_j)}{d\phi_j}\biggr)^2\\
&=\frac{1}{2}\{t,\phi_j\}+N_j\biggl(\frac{dt}{d\phi_j}\biggr)^2\frac{(2mn-N_j)(p_j'(t))^2-2mn\cdot p_j''(t)p_j(t)}{\bigl(2mn\cdot p_j(t)\bigr)^2}\,,
\end{aligned}
\end{equation}
where~$\{t,\phi_j\}$ is the Schwarzian derivative of~$t$ with respect to~$\phi_j$. It is well known that~$\{t,\phi_j\}=2\Bigl(\frac{dt}{d\phi_j}\Bigr)^2\mathcal{Q}_j(t)$, where~$\mathcal{Q}_j$ is as in~\eqref{eqn:Qform} with parameters~$\alpha, \beta$ and~$\gamma$ replaced by
\[
\alpha_j\=\frac{1}{2}\biggl(1+\frac{r_j}{m}-\frac{k_j}{n}\biggr),\quad\beta_j\=\frac{1}{2}\biggl(1-\frac{k_j}{n}-\frac{r_j}{m}\biggr)\quad\gamma_j=1
\]
respectively (see~\eqref{eqn:hypar}). A computation shows that
\[
\mathcal{Q}_j(t)+N_j\frac{(2mn-N_j)(p_j'(t))^2-2mn\cdot p_j''(t)p_j(t)}{\bigl(2mn\cdot p_j(t)\bigr)^2}=-\biggl(\frac{N_j}{2mn}\biggr)\frac{1}{t^2(1-t)}\,.
\]
Since~$\Bigl(\frac{dt}{d\phi_j}\Bigr)^2\frac{1}{t^2(1-t)}=Q_j^{m-2r_j}R_j^{n-2k_j}$, the above computation and~\eqref{eqn:Pj} prove the relation for~$P_j$.
A purely modular proof of the same fact can also be given. The first line of~\eqref{eqn:Pj} can be rewritten as follows
\begin{equation}
\label{eqn:rcD}
\frac{d P_j}{d\phi_j}-P_j^2=\frac{N_j^2}{(2mn)^2}\frac{[\Delta_j,\Delta_j]_2}{(\frac{2mn}{N_j}+1)\Delta_j^2}\,,
\end{equation}
where~$\frac{N_j^2[\Delta_j,\Delta_j]_2}{2mn+N_j}=2mn\frac{d^2\Delta_j}{d\phi_j^2}-(2mn+N_j)\bigl(\frac{d\Delta_j}{d\phi_j}\bigr)^2$. 
The right-hand side of~\eqref{eqn:rcD} is a twisted modular form of pure~$j$-th weight~$4$ with trivial multiplier system (see the discussion after the proof on multiplier systems). One can observe that the only twisted modular form of pure~$j$-th weight~$4$ and trivial multiplier system is given by~$Q_j^{m-2r_j}R_j^{n-2k_j}$ (because there is only one cusp), and deduce the factor~$N_j^2/(2mn)^2$ by comparing the Fourier expansion at~$\infty$ of this modular form with the one in~\eqref{eqn:rcD}.}
\item
It follows from the previous point and the examples in Section~\ref{sec:tmf} that~$Q_j$ and~$R_j$ are twisted modular forms of pure $j$-th weight~$2n/N_j$ and~$2m/N_j$ respectively.

They are algebraically independent. To prove this, consider~$Q_j^{mN_j}$ and~$R_j^{nN_j}$; they have the same integral weight and so transform in the same way with respect the action of~$\Delta(n,m,\infty)$. Moreover they are linearly independent, since the zeros of~$Q_j^m$ are concentrated in~$e_m$ and the zeros of~$R_j^n$ are concentrated in~$e_n$. A standard trick (see the proof of Proposition 4, pag. 15 of~\cite{bhgz}) shows that~$Q_j^{mN_j}$ and~$R_j^{nN_j}$, and so~$Q_j,R_j$, are algebraically independent.

In order to prove that~$\C[Q_j,R_j]$ is the space on the right-hand side of~\eqref{eqn:freealg}, we compute the dimension of the graded part~$\C[Q_j,R_j]_w$ and compare it with~\eqref{eqn:hydim} for~$\vec{w}=(0,\dots,0,w_j=w,0,\dots,0)$.
A basis of~$\C[Q_j,R_j]_w$ is given by monomials of the form~$Q_j^aR_j^b$ where~$a\cdot2n/N_j+b\cdot2m/N_j=w$. The number of such monomials is the number of non-negative solutions~$(a,b)$ of~$a\cdot2n+b\cdot 2m=wN_j$.
Let~$(a_0,b_0)$ be such a solution with~$a_0$ minimal. This condition implies in particular that~$a_0<m$. All the other solutions are of the form~$(a_0+xm,b_0-xn)$ for~$x\in\Z_{\ge0}$ with~$2n(a_0+xm)\le N_jw$. It follows that~$x\le wN_j/(2mn)-a/m\in\Z_{\ge0}$. A simple analysis shows that the maximum of such~$x$ is~$\lfloor\tfrac{wN_j}{2mn}\rfloor$ or~$\lfloor\tfrac{wN_j}{2mn}\rfloor-1$ depending on~$2na_0\equiv wN_j \mod 2mn$ holds or not. Accordingly, the number of solutions is~$1+\lfloor\tfrac{wN_j}{2mn}\rfloor$ or~$\lfloor\tfrac{wN_j}{2mn}\rfloor$, as in~\eqref{eqn:hydim}.
\end{enumerate}
\end{proof}

\begin{rem} \label{rem XYZ}
\red{The first part of Theorem~\ref{thm:hymain} could be deduced by applying the following transformations to the functions~$X_j,Y_j,Z_j$ of the Ohyama system~\eqref{eqn:ohyama} associated to~$\mathcal{T}\bigl(\tfrac{k_j}{n},\tfrac{r_j}{m}\bigr)$
\[
\begin{aligned}
P_j&\:=\frac{1}{2mn}\Bigl(n(m-r_j)X_j+(mk_j+nr_j)Y_j+m(n-k_j)Z_j\Bigr)\,,\\
Q_j&\:=\Bigl((X_j-Y_j)^{k_j}(Z_j-Y_j)^{n-k_j}\Bigr)^{1/N_j}\,,\quad R_j\:=\Bigl((X_j-Y_j)^{m-r_j}(Z_j-Y_j)^{r_j}\Bigr)^{1/N_j}\,.
\end{aligned}
\]
The fact that~$P_j,Q_j$ and~$R_j$ are holomorphic on~$\Po$ follows by observing that
\begin{equation*}
\label{eqn:hyres}
\begin{aligned}
\Res_{e_n}(X_j)&=\frac{n-k_j}{2}\,,\quad\Res_{e_n}(Y_j)=\frac{n-k_j}{2}\,,\quad\Res_{e_n}(Z_j)=-\frac{n+k_j}{2}\,,\\
\Res_{e_m}(X_j)&=-\frac{m+r_j}{2}\,,\quad\Res_{e_m}(Y_j)=\frac{m-r_j}{2}\,,\quad\Res_{e_m}(Z_j)=\frac{m-r_j}{2}\,.
\end{aligned}
\end{equation*}
The modularity properties follow from the well known identity~$y_j(\tau)=\frac{t'(\tau)}{\phi_j'(\tau)}$, where~$y_j$ is a solution of the~$Q$-form of the HGDE and its relation with~$X_j,Y_j$ and~$Z_j$ is given in~\eqref{eqn:XYZ}.}
\end{rem}

We discuss a few consequences of the above theorem. In order to do this, we need to introduce multiplier systems first.
Fix a branch of the logarithm. The automorphy factor~$J_\phi$ of weight~$\vec{w}=(w_1,\dots,w_h)\in\Q^h$ defined in Section~\ref{sec:tmf} can be represented as
\[
J_\phi(\gamma,\tau)\=v(\gamma)\prod_{j=1}^h{(c_j\phi_j(\tau)+d_j)^{w_j}}\,\quad\gamma\in\Delta(n,m,\infty)\,,
\]
where~$v\colon\Delta(n,m,\infty)\to\C$ is a function with~$|v|=1$ called~\emph{multiplier system}. Thanks to the cocycle relation, it is enough to specify the values of~$v$ at the generators~$T,S$ of~$\Delta(n,m,\infty)$ to determine it completely; in particular~$v$ is not in general a group homomorphism.  Since the functions we consider are~$T$-invariant, we always have $v(T)=1$ and we need only to specify the value of~$v$ in~$S$.

The modular forms related to system~\eqref{eqn:hyrc} are of pure~$j$-th weight~$w=\hat{w}/N_j$ where~$\hat{w}\in\Z_{\ge0}$. It follows that we can describe them only in terms of powers of~$J_{\phi_j}(\gamma,\tau)^{1/N_j}$ (see the definition of~$J_{\phi_j}$ in~\eqref{eqn:Jj}) and multiplier systems; the space of modular forms of pure~$j$-th weight~$w$ and multiplier system~$v$ will be denoted by~$M_{J_{\phi_j}^w}\bigr(\Delta(n,m,\infty),\phi,v(S)\bigr)$.
\begin{coro}[Dimension formula]
\label{cor:dim}
Let~$\vec{w}_j\in\Q^h$ be of pure $j$-th weight~$w$. Recall that~$M_{\vec{w}_j}(\Delta(n,m,\infty),\phi)$ is the space of modular forms with respect to all automorphy factors of fixed weight~$\vec{w}_j$. Then
\[
\begin{aligned}
\bigoplus_{\vec{w}_j\in\Q^h}{M_{\vec{w}_j}\bigl(\Delta(n,m,\infty),\phi\bigr)}&\=\bigoplus_{\overset{(a,b)\in\Z_{\ge0}}{2an+2bm=wN_j}}M_{J_{\phi_j}^w}\Bigl(\Delta(n,m,\infty),\phi,e^{-\frac{2\pi i}{N_j}(ak_j+b(m-r_j))}\Bigr)\\
&\=\Biggl\{f\in\mathcal{H}_0(\Po):
\begin{aligned}
&f(T\tau)=f(\tau)\text{ and }\\
&f(S_n\tau)=f(\tau)\phi_j(\tau)^we_n^{-\bigl(2\mathrm{ord}_{e_n}(f)+k_jw\bigr)}\;
\end{aligned}
\Biggr\}\,,
\end{aligned}
\]
where~$\mathcal{H}_0(\Po)$ denotes the set of complex-valued holomorphic function on~$\Po$ of moderate growth at~$\infty$.
Moreover, we have the dimension formula
\[
\dim\bigoplus_{\vec{w}_j\in\Q^h}{M_{\vec{w}_j}\bigl(\Delta(n,m,\infty),\phi\bigr)}\=
\begin{cases}
1+\Bigl\lfloor\frac{wN_j}{2mn}\Bigr\rfloor &\text{if }\Bigl\{\tfrac{wN_j}{2mn}\Bigr\}=\tfrac{a}{n}+\frac{b}{m}\\
\Bigl\lfloor\frac{wN_j}{2mn}\Bigr\rfloor & \text{otherwise},
\end{cases}
\]
where~$\{x\}\in[0,1)$ is the fractional part and~$a,b\in\Z_{\ge0}$.
\end{coro}
\begin{proof}
The multiplier systems~$v_{Q_j},v_{R_j}$, associated to~$Q_j,R_j$ respectively, are easily determined by considering the fixed points of the action of~$\Delta(n,m,\infty)$. For instance, the relation
\[
Q_j(e_n)\=Q_j(Se_n)\=v_{Q_j}(S)\phi_j(e_n)^{2n/N_j}Q_j(e_n)\,,
\]
together with the identity~$\phi_j(e_n)=e^{\pi ik_j/n}$, imply that~$v_{Q_j}(S)=e^{(-2\pi i k_j)/N_j}$.
Similarly, by considering the transformation~$T^{-1}S$, which fixes~$e_m$, we can compute the multiplier system for~$R_j$. It is given by~$v_{R_j}(S)\=e^{-2\pi i(m-r_j)/N_j}$.
It follows that the multiplier system of a modular form of pure~$j$-th weight~$w$, that by Theorem~\ref{thm:hymain} is a weighted homogeneous polynomial in~$Q_j,R_j$, is of the form~$v_{Q_j}(S)^av_{R_j}(S)^b=e^{-2\pi i(ak_j+b(m-r_j))/N_j}$ if~$2(an+bm)=w$. This proves the first identity.

To prove the second identity, it is enough to check that the generators~$Q_j$ and~$R_j$ belong to the set. This is an immediate consequence of the the identities~$\mathrm{ord}_{e_n}(Q_j)=0$ and~$\mathrm{ord}_{e_n}(R_j)=1$ proved in Theorem~\ref{thm:hymain}. By a dimension argument (any~$f$ in the set is subject to the valence formula~\eqref{eqn:hydeg}, and then to the bounds in Proposition~\ref{prop:dimtws}) we see that~$Q_j$ and~$R_j$ generate the set; this concludes the proof of the second identity.

The dimension formula is proven in the second part of the proof of Theorem~\ref{thm:hymain}.
\end{proof}

\begin{rem}{\rm
Theorem~\ref{thm:hymain} was known in the case~$n=2$ and~$k=r=1$ since the work of Hecke~\cite{Hecke} (see also~\cite[Chapter 5]{BerndtKnopp}). The second description of the modular space in Corollary~\ref{cor:dim} is also a generalization of the one used by Hecke in his book.}
\end{rem}
\begin{rem}{\rm
In~\cite{dgms}, Doran et al. gave basis for the space of modular forms of even integral weight~$2k$ for the groups~$\Delta(n,m,\infty)$ in terms of certain functions~$E^{(1)}_{2l,\mathfrak{t}},3\le l\le n$ and~$E^{(2)}_{2l,\mathfrak{t}},2\le l\le m$. These functions are constructed from solutions of the Halphen system (see~\cite[Equation (1.2) ]{dgms}) for a suitable choice of parameters.
By comparing the divisors, we can express the functions~$E^{(1)}_{2l,\mathfrak{t}},E^{(2)}_{2l,\mathfrak{t}}$ in terms of the functions~$Q=Q_1,R=R_1$ in Theorem~\ref{thm:hymain} as follows
\[
\begin{aligned}
E^{(1)}_{2l,\mathfrak{t}}&\=Q^{ml-m-l}R^{n-l}\,,\quad 3\le l\le n\,,\\
E^{(2)}_{2l,\mathfrak{t}}&\=Q^{m-l}R^{nl-n-l}\,,\quad2\le l\le m\,.
\end{aligned}
\]
In particular, the algebraic relations between the~$E^{(1)}_{2l,\mathfrak{t}},E^{(2)}_{2l,\mathfrak{t}}$ in~\cite{dgms} can be easily determined from the above identities, $Q,R$ being algebraically independent. Moreover, differential relations between these functions can be determined from the above identities and the system in Theorem~\ref{thm:hymain}.}
\end{rem}

In the non-twisted case~($j=1$), the elements~$Q=Q_1,R=R_1,$ and~$\Delta=\Delta_1$ determine a Rankin-Cohen structure on~$\C[Q,R]$. More precisely, from Theorem~\ref{thm:hymain} and Theorem~\ref{thm 1} we immediately get the following result.
\begin{coro}[RC structures]
\label{cor:RC}
Let~$n\le m$ be such that~$1/n+1/m<1$. System~\eqref{eqn:hyrc} in the case~$j=1$ reduces to the RRC system
\begin{equation}
\label{eqn:rcmod}
\left\{
\begin{aligned}
P'&\=P^2\-\Bigl(\frac{mn-m-n}{2nm}\Bigr)^2Q^{m-2}R^{n-2}\\
Q'&\=\frac{2n}{mn-m-n}PQ\-\frac{R^{n-1}}{m}\\
R'&\=\frac{2m}{mn-m-n}PR\-\frac{Q^{m-1}}{n}\,.
\end{aligned}
\right.
\end{equation}
which defines a Rankin-Cohen structure on the space~$\C[Q,R]$.
This has the form described in Remark~\ref{rmk:Delta} by setting~$t_1=P=\Delta'/\Delta, t_2=Q, t_3=R$, and~$F=\Delta$.
\end{coro}

We stress that the Rankin-Cohen structure~$\bigl(\C[Q,R],[\,,]_n\bigr)$, with brackets induced by system~\eqref{eqn:rcmod}, is in general not the classical RC structure on a ring of modular forms, because the generators~$Q,R$ (and $\Delta$) are associated to different automorphy factors and~$\C[Q,R]$ is not a usual ring of modular forms. However, the brackets on~$\C[Q,R]$ being induced by the derivation~$d/d\tau$, their restriction to classical subrings of modular forms~$M_{J}(\Delta(n,m,\infty))\subset\C[Q,R]$ coincide with the RC brackets defined by Cohen.

\begin{rem}{\rm
Theorem~\ref{thm:hymain} does not give an RRC system with derivation~$d/d\tau$ in the case~$j>1$. The reason is that~$\C[Q_j,R_j]$ is the space of modular forms of pure~$j$-th weight, but the differentiation~$d/d\tau$ does not preserves the pure~$j$-th weight if~$j>1$ (see the computation of~$\phi_j'$ in Section~\ref{sec:tmf}).
On the other hand, the natural derivation considered in~\eqref{eqn:hyrc}, that is $d/d\phi_j=(\phi_j')^{-1}d/d\tau$, preserves the pure~$j$-th weight, but it does not preserve holomorphicity (for instance, in the equation for~$\red{dQ_j/d\phi_j}$ one gets~$R_j^{n-k_j}Q_j^{1-r_j}$ that is in general not holomorphic). In this case, one can define a RC structure on~$\C[Q_j,R_j]$ from Theorem~\ref{thm:hymain} at the cost of inverting the functions~$Q_j$ and~$R_j$. }
\end{rem}

The last corollaries of Theorem~\ref{thm:hymain} concern the Gauss hypergeometric function
\begin{equation}
\label{eqn:ghg}
F(\alpha,\beta;\gamma;t)\:=\sum_{n=0}^\infty{\frac{(\alpha)_n(\beta)_n}{(\gamma)_n(n!)^2}t^n}\,,
\end{equation}
where~$(x)_n:=x(x+1)\cdots(x+n-1)$.
\begin{coro}[\red{Inversion formulae}]
\label{cor:hgde}
Let~$n,m,k_j,r_j,$ and~$Q_j,R_j$ be as in Theorem~\ref{thm:hymain} and let~$N_j:=mn-nr_j-mk_j$.
Then
\[
F\biggl(\frac{N_j+2nr_j}{2mn},\frac{N_j}{2mn};1;\frac{Q_j(\tau)^m-R_j(\tau)^n}{Q_j(\tau)^m}\biggr)\=Q_j(\tau)^{N_j/2n}\,.
\]
\end{coro}
\begin{proof}
\red{The function $F\bigl(\frac{N_j+2nr_j}{2mn},\frac{N_j}{2mn};1;z\bigr)$ is a holomorphic solution in~$z=0$ of the differential equation in~\eqref{eqn:hgde} with parameters~$\alpha=\frac{N_j+2nr_j}{2mn}$, $\beta=\frac{N_j}{2mn}$, and~$\gamma=1$. From~\eqref{eqn:y} it follows that
\begin{equation}
\label{eqn:cor1}
y_j(z)\=z^{\frac{1}{2}}(1-z)^{\frac{n-k_j}{2n}}F\biggl(\frac{N_j+2nr_j}{2mn},\frac{N_j}{2mn};1;z\biggr)
\end{equation}
is a solution of the~$\mathcal{Q}$-form of the hypergeometric differential equation. It is well known that~$y_j(\tau)^2=c\frac{t'(\tau)}{\phi_j'(\tau)}$, where~$t$ is a Hauptmodul for the triangle group~$\Delta(n,m,\infty)$ and~$\phi=(\phi_j)_j$ the corresponding modular embedding (this follows from the fact that the Wronskian of the~$\mathcal{Q}$-form of the HGDE is a constant~$c$).
From the description of~$Q_j$ and~$R_j$ in terms of~$t$ and~$\phi_j$ in Theorem~\ref{thm:hymain} we have~$t=(Q_j^m-R_j^n)/Q_j^m$; by combining then~\eqref{eqn:cor1} and the expression for~$y_j(\tau)^2$ we get
\[
F\biggl(\frac{N_j+2nr_j}{2mn},\frac{N_j}{2mn};1;\frac{Q_j(\tau)^m-R_j(\tau)^n}{Q_j(\tau)^m}\biggr)^2=c\frac{t'(\tau)}{\phi_j'(\tau)t(\tau)(1-t(\tau))^{\frac{n-k_j}{n}}}=c\cdot Q_j(\tau)^{N_j/n}\,.
\]
and~$c=1$ follows by comparing the~$q$-expansion of both sides.}
\end{proof}
In the case~$k_1=r_1=1$ the statement is simply giving the solution of the uniformizing differential equation for~$\Po/\Delta(n,m,\infty)$ with respect to the Hauptmodul~$(Q^m-R^n)/Q^m$. As an example, in the case~$(n,m)=(2,3)$ Corollary~\ref{cor:hgde} specializes to the classical formula
\[
F\biggl(\frac{1}{12},\frac{5}{12};1;\frac{1728}{j}\biggr)\=E_4^{1/4}\,.
\]
where~$j$ is the~$j$-invariant and~$E_4$ the normalized weight 4 Eisenstein series on~$\mathrm{SL}_2(\Z)$.

In Corollary~\ref{cor:hgde} the rational parameters~$\alpha,\beta$ of the hypergeometric function~\eqref{eqn:ghg} depend on~$k_j,r_j$, while the parameter~$(Q_j^m-R_j^n)/Q_j^m$ does not. Set~$Q:=Q_1,R:=R_1$. Since~$\frac{Q_j^{m}-R_j^n}{Q_j^m}=\frac{Q^{m}-R^n}{Q^m}$, Corollary~\ref{cor:hgde} then let us evaluate the hypergeometric function~\eqref{eqn:ghg} in different points~$\alpha,\beta$ with respect to the same parameter~$(Q^m-R^n)/Q^m$. For instance, when~$(n,m)=(2,5)$ we have
\[
F\biggl(\frac{7}{20},\frac{3}{20};1;\frac{Q(\tau)^5-R(\tau)^2}{Q(\tau)^5}\biggr)^2=Q(\tau)^\frac{3}{2}\,,\quad F\biggl(\frac{9}{20},\frac{1}{20};1;\frac{Q(\tau)^5-R(\tau)^2}{Q(\tau)^5}\biggr)^2=\frac{Q(\tau)^{\frac{3}{2}}}{\phi'(\tau)}\,,
\]
corresponding to the possible values~$(k,r)=(1,1)$ and~$(k,r)=(1,2)$ respectively. More generally, we have the following result.
\begin{coro}[\red{Inversion formulae II}]
\label{cor:hgde'}
Let~$n,m,k_j,r_j$ and~$(\phi_j)_j$ be as in Theorem~\ref{thm:hymain} and denote~$Q=Q_1,R=R_1$. Then, for~$j=1,\dots,h,$
\[
F\biggl(\frac{N_j+2nr_j}{2mn},\frac{N_j}{2mn};1;\frac{Q(\tau)^m-R(\tau)^n}{Q(\tau)^m}\biggr)^2\=\frac{Q(\tau)^{\tfrac{N_j}{n}+r_j-1}R(\tau)^{k_j-1}}{\phi_j'(\tau)}\,.
\]
\end{coro}

\subsection{Examples}
\label{sec:examples}
\subsubsection{Arithmetic example: the group~$\Delta(3,3,\infty)$}
In this section we discuss the RRC system associated to the group~$\Delta(3,3,\infty)$, \red{which is a congruence subgroup of~$\SL2$ of level~$2$} generated by~$T=\left(\begin{smallmatrix}1 & 2 \\0 &1\end{smallmatrix}\right)$ and~$S=\left(\begin{smallmatrix}1 & -1 \\1 & 0\end{smallmatrix}\right)$.
We remark that in~\cite{zud03} (different) Ramanujan systems associated to arithmetic triangle groups are constructed. It is straightforward to relate the solutions of the systems in~\cite{zud03} to the ones of Theorem~\ref{thm:hymain}.

The specialization of Theorem~\ref{thm:hymain} to~$n=m=3$ implies that the RRC system associated to~$\Delta(3,3,\infty)$ is
\begin{equation}
\label{eqn:s33}
\left\{
\begin{aligned}
P'&\=P^2\-\frac{QR}{36}\\
Q'&\=2PQ\-\frac{R^2}{3}\\
R'&\=2PR\-\frac{Q^2}{3}
\end{aligned}
\right.\,,
\end{equation}
\red{where~$'=(\pi i)^{-1}\tfrac{d}{d\tau}$.}
\red{Solutions of this system can be constructed from the solutions of the hypergeometric differential equation~\eqref{eqn:hgde} with parameters~$\alpha=1/2,\,\beta=1/6,$ and~$\gamma=1$. If~$y,\hat{y}$ are two Frobenius solutions at the singular point~$0$ (normalized by~$y(0)=1$ and~$\hat{y}=\log(t)y(t)+\tilde{y}(t)$ with~$\tilde{y}(0)=0$), the parameter at~$\infty$ given by~$\tilde{q}:=\exp(\hat{y}/y)=t\exp(\tilde{y}/y)$ is related to the parameter~$q:=e^{\pi i\tau}$, where $\tau\in\Po$, by the formula~$\tilde{q}=48\sqrt{-3}\cdot q$. By inverting the power series for~$\tilde{q}$ to find the expansion of~$t$ in~$\tilde{q}$, and considering the relation between~$\tilde{q}$ and~$q$, it follows that the~$q$-expansions of the solutions~$P_1=P$, $Q=Q_1$ and~$R=R_1$ given in Theorem~\ref{thm:hymain} are
\[
\begin{aligned}
P(\tau)&=\frac{1}{6}\-4q^2\-12q^4\-16q^6\+\cdots\\
Q(\tau)&=1\+8\sqrt{-3}q\+24q^2\+32\sqrt{-3}q^3\+24q^4\+48\sqrt{-3}q^5\+\cdots\,,\\
R(\tau)&=1\-8\sqrt{-3}q\+24q^2\-32\sqrt{-3}q^3\+24q^4\-48\sqrt{-3}q^5\+\cdots\,.\
\end{aligned}
\]
A bit of experimentation shows that}
\begin{equation}
\label{eqn:a33}
P\=\frac{E_2}{6}\,,\quad Q\=\theta_3^4\+2\sqrt{-3}\cdot\theta_3^2\theta_2^2\+\theta_2^4\,,\quad R\=\theta_3^4\-2\sqrt{-3}\cdot\theta_3^2\theta_2^2\+\theta_2^4\,.
\end{equation}
where~\red{$\theta_2$ and~$\theta_3$} are the classical theta series
\[
\theta_3(\tau)\:=\sum_{n\in\Z}{q^{n^2}}\,,\quad\theta_2(\tau)\:=\sum_{n\in\Z+1/2}{q^{n^2}}\,.
\]
\red{From the proof of Corollary~\ref{cor:dim} we know that~$Q$ and~$R$ are modular forms of weight~$2$ with multiplier system given by~$v_Q(S)=e^{\frac{4\pi i}{3}}$ and~$v_R(S)=e^{\frac{2\pi i}{3}}$ respectively. Moreover, their restriction to the principal congruence subgroup~$\Gamma(2)$ gives modular forms with trivial multiplier system (this can be shown by looking at the value of~$v_Q(G)$ and~$v_R(G)$ where~$G=S^{-2}T^{-2}S^{-1}$ is a generator of~$\Gamma(2)$ together with~$T$).}
We conclude then that any modular form of rational weight on~$\Delta(3,3,\infty)$ is a polynomial in the combination of unary theta series~$Q,R$.

In~\cite{zud03}, functions~$A_1(q),\dots,A_r(q)\in\C[q]$ are called~\emph{Ramanujan functions} if they satisfy the following properties: they are algebraically independent over~$\C(q)$, they satisfy an algebraic system of differential equations, and their Taylor coefficients with respect to~$q$ are integral and of polynomial growth.
From Theorem~\ref{thm:hymain} and the explicit description~\eqref{eqn:a33} it follows that~$P,Q,R$ satisfy all these assumptions but the integrality of the coefficients. However, it is not hard to construct Ramanujan functions from~\eqref{eqn:a33}. The functions
\[
\hat{P}\:=P\,,\quad \hat{Q}\:=\frac{Q+R}{2}\=\theta^4+\theta_F^4\,,\quad\hat{R}\:=QR=E_4\,,
\]
have integral~$q$-expansion and it follows immediately from~\eqref{eqn:s33} that~$\hat{P},\hat{Q},\hat{R}$ satisfy the~RRC system
\[
\left\{
\begin{aligned}
\hat{P}'&\=\hat{P}^2\-\frac{\hat{R}}{36}\\
\hat{Q}'&\=2\hat{P}\hat{Q}\-\frac{\hat{Q}^2-2\hat{R}}{3}\\
\hat{R}'&\=4\hat{P}\hat{R}\-\frac{\hat{Q}^3-3\hat{Q}\hat{R}}{3}
\end{aligned}
\right.
\]
The remaining arithmetic cases for groups of the form~$\Delta(n,m,\infty)$ can be handled similarly.

\subsubsection{The group~$\Delta(2,5,\infty)$}
\label{sec:25}
In the case~$(n,m)=(2,5)$ we have~$h=2$ and~$N_1=3,N_2=1$. Let~$(1,\phi)$ denote the modular embedding and let~$B^2:=\phi'/Q_2$ be a holomorphic twisted modular form of weight~$(2,-6)$. The systems we get from Theorem~\ref{thm:hymain} are
\begin{equation}
\label{eqn:ss25}
\left\{
\begin{aligned}
P_1'&\=P_1^2\-\Bigl(\frac{3}{20}\Bigr)^2Q_1^3\\
Q_1'&\=\frac{4}{3}P_1Q_1\-\frac{R_1}{5}\\
R_1'&\=\frac{10}{3}P_1R_1\-\frac{Q_1^4}{2}
\end{aligned}
\right.
\qquad\qquad
\left\{
\begin{aligned}
P_2'&\=P_2^2\phi'\-\Bigl(\frac{1}{20}\Bigr)^2Q_2\phi'\\
Q_2'&\=4P_2Q_2\phi'\-\frac{R_2B^2}{5}\\
R_2'&\=10P_2R_2\phi'\-\frac{Q_2^3\phi'}{2}\,.
\end{aligned}
\right.
\end{equation}
We know that~$M_{(*,0)}(\Delta(2,5,\infty),\phi)=\C[Q_1,R_1]$ and~$M_{(0,*)}(\Delta(2,5,\infty),\phi)=\C[Q_2,R_2]$, but these twisted modular forms are not enough to generate the whole space~$M_{(*,*)}(\Delta(n,m,\infty))$.
For instance, consider the full Hilbert modular group~$\Gamma_5:=\mathrm{SL}_2(\mathcal{O}_5)$ associated to the ring of integers~$\mathcal{O}_5$ of the field~$\Q(\sqrt{5})$. The group~$\Delta(2,5,\infty)$ embeds into this \red{Hilbert} modular group via the real embeddings of its trace field~$\Q(\sqrt{5})$. As is well known~(\cite[Part II, Theorem 1.40]{bhgz}), there exists a Hilbert cusp form~$s_5$ of parallel weight~$(5,5)$ on~$\Gamma_5$ whose zero locus is the diagonal in the Hilbert modular surface~$\Po^2/\Gamma_5$. This implies that the restriction of~$s_5$ to the embedded curve~$\Po/\Delta(2,5,\infty)$ is not zero and defines a twisted cusp form of weight~$(5,5)$. This cusp form cannot be generated from the solutions~$Q_1,R_1,Q_2,R_2$ of the above systems. Nevertheless, the next proposition shows that we can actually construct the whole ring of integral weight twisted modular forms on~$\Delta(n,m,\infty)$ starting from systems~\eqref{eqn:ss25}.
\begin{proposition}
There exists a holomorphic twisted modular form~$B$ of weight~$(1,-3)$ on~$\Delta(2,5,\infty)$ such that
\[
M_{(\Z,\Z)}(\Delta(n,m,\infty))\=\C[Q_2,R_2,B]\,,
\]
where~$M_{(\Z,\Z)(\Delta(n,m,\infty))}$ denotes the space of twisted modular forms of integral weight.
The action of~$d/d\tau$ on~$M_{(*,*)}(\Delta(n,m,\infty))$ is described by the system of ODEs
\begin{equation}
\label{eqn:s25}
\left\{
\begin{aligned}
P_1'&\=P_1^2\-\Bigl(\frac{3}{20}\Bigr)^2Q_2^3B^4\\
P_2'&\=P_2B^2Q_2\cdot P_2\-\Bigl(\frac{1}{20}\Bigr)^2Q_2^2B^2\\
B'&=(P_1-3P_2B^2Q_2)\cdot B\\
Q_2'&\=4P_2B^2Q_2\cdot Q_2\-\frac{R_2B^2}{5}\\
R_2'&\=10P_2B^2Q_2\cdot R_2\-\frac{Q_2^4B^2}{2}\,.
\end{aligned}
\right.
\end{equation}
\end{proposition}
\begin{proof}
The relevant part of the proof concerns the construction of~$B$.
As mentioned above, there exists a twisted cusp form of weight~$(5,5)$ on~$\Delta(2,5,\infty)$. The product~$\phi'Q_2$ defines a twisted modular for of weight~$(2,2)$ with a double zero in~$e_5$, since both~$Q_2$ and~$\phi'$ have a simple zero in~$e_5$. From the normalization of the Hauptmodul~$t$ of~$\Delta(n,m,\infty)$ in~\eqref{eqn:divhaupt} we see that~$t\cdot(\phi'Q_2)^3$ has a zero at~$\infty$ and one in~$e_5$. It follows that~$A:=t\cdot(\phi'Q_2)^3/s_5$ is a holomorphic twisted modular form of weight~$(1,1)$ with a unique zero in~$e_5$. Finally, we define~$B:=\phi'/A$. It is by construction holomorphic and of weight~$(1,-3)$. We find in particular that~$B^2Q_2=\phi'$, and from this relation we can deduce system~\eqref{eqn:s25} from systems~\eqref{eqn:ss25}.

Since~$Q_2$ and~$R_2$ are algebraically independent from Theorem~\ref{thm:hymain}, it is immediate to prove that~$Q_2,R_2$ and~$B$ are also algebraically independent ($B$ is the unique form of mixed weight among the three).
To prove that~$\C[Q_2,R_2,B]$ is the whole space of twisted modular forms of integral weight we argue like in the proof of Theorem~\ref{thm:hymain}: we compute the dimension of the graded part of the polynomial ring and compare it with the upper bound given in Proposition~\ref{prop:dimtws}. An element of weight~$(k,l)$ in~$\C[Q_2,R_2,B]$ comes from monomials of the form~$B^k\cdot Q_2^a\cdot R_2^b$ with~$4a+10b=l+3k$. The dimension of the graded piece~$\C[Q_2,R_2,B]_{(k,l)}$ is then the number of pairs~$(a,b)\in\Z_{\ge0}^2$ with~$4a+10b=l+3k$. The same argument as in the second part of the proof of Theorem~\ref{thm:hymain} shows that it is precisely the number predicted by the bound in Proposition~\ref{prop:dimtws}.
\end{proof}

We conclude this discussion with an observation on the RC structure.
System~\eqref{eqn:s25} does not fall into the case described in Theorem~\ref{thm 1}, because the space of twisted modular forms on~$\Delta(2,5,\infty)$ is bi-graded. However, it shares the main characteristics of RRC systems: if we write~$B^2Q_2=\phi'$,  each equation in~\eqref{eqn:s25} is of the form
\begin{equation}
f'\=(kP_1+lP_2\phi')\cdot f+p_f(Q_2,R_2,B)
\end{equation}
where~$(k,l)$ is the weight of the twisted modular form~$f$ and~$p_f$ is a weighted homogeneous polynomial of weight~$(k+2,l)$. In other words, we have a Ramanujan-Serre derivation on~$M_{(\Z,\Z)}(\Delta(n,m,\infty))$ given by
\[
\partial f := f'\-(kP_1+l\phi'P_2)f\;\;\in M_{(k+2,l)}(\Delta(n,m,\infty))\,
\]
if~$f$ is of weight~$(k,l)$.
This observation suggests that the RRC systems we introduced may have a natural generalization to systems of first order nonlinear ODEs whose space of solutions has naturally a structure of multi-graded algebra.

\subsection*{Acknowledgement}{The first author is supported by the LOEWE research unit USAG, and by the Deutsche Forschungsgemeinschaft (DFG) through the Collaborative Research Centre TRR 326 “Geometry and Arithmetic of Uniformized Structures”, project number 444845124.
The second author started and concluded the major part of his results during his one-year visit to Max Planck Institute for Mathematics (MPIM) Bonn. So, he would like to thank MPIM and its staff for preparing such an excellent ambience for doing mathematical work. 
Finally, the authors wants to thank the anonymous referees for many valuable suggestions that helped to improve the paper.}

\bibliography{ref}{}
\bibliographystyle{plain}
\end{document}